\documentclass[11pt, reqno]{amsart}
\usepackage{amsmath,amssymb,amsthm,amscd,amsopn,amsfonts,amsxtra} 
\usepackage{latexsym,array,tensor,verbatim,hyperref,
mathtools} 

\usepackage[mathscr]{eucal}
\usepackage{eucal} 

\usepackage{inputenc,enumitem}

\usepackage{graphicx}

\usepackage{epsfig,pstricks} 
\usepackage{psfrag,color,float,
caption,subcaption}
\restylefloat{figure}

\newtheorem{theorem}[subsection]{Theorem} 
\newtheorem{lemma}[subsection]{Lemma}
\newtheorem{corollary}[subsection]{Corollary}
\newtheorem{proposition}[subsection]{Proposition} 
\newtheorem{prop}[subsection]{Proposition}
\newtheorem{example}[subsection]{Example}

\theoremstyle{definition}
\newtheorem{definition}{Definition}[section]

\theoremstyle{remark}
\newtheorem{remark}[subsection]{Remark}

\theoremstyle{hypotheis}
\newtheorem{hypo}{Hypothesis}

\newcommand{\norm}[1]{\lVert#1\rVert }

\definecolor{orange}{rgb}{0.995, 0.75, 0.35}
\definecolor{purple}{rgb}{0.7, 0.2, 0.5}
\definecolor{royalblue}{rgb}{0.2, 0.7, 0.8}
\definecolor{darkgreen}{rgb}{0.2,0.725,0.25}

\hypersetup{colorlinks,
linkcolor=royalblue,
filecolor=darkgreen,
urlcolor=royalblue,
citecolor=royalblue
}

\newcommand{\be}{\begin{equation}}
\newcommand{\ee}{\end{equation}}
\newcommand{\ba}{\begin{array}}
\newcommand{\ea}{\end{array}}
\newcommand{\bea}{\begin{eqnarray}}
\newcommand{\eea}{\end{eqnarray}}
\newcommand{\beas}{\begin{eqnarray*}}
\newcommand{\eeas}{\end{eqnarray*}}

\def\al{\alpha}
\def\de{\delta}
\def\eps{\epsilon}
\def\ga{\gamma}
\def\lam{\lambda}
\def\om{\omega}
\def\veps{\varepsilon}
\def\vphi{\varphi}
\def\De{\Delta}
\def\Ga{\Gamma}

\def\Om{\Omega}
\def\mx{\mathbf{x}}
\def\mOmega{\mathbf{\Omega}}
\def\mL{\mathbf{L}}

\global\long\def\sgn{\operatorname{sgn}}

\usepackage{mathrsfs}   

\DeclareMathOperator{\dive}{div}

\def\cD{\mathcal{D}}
\def\cE{\mathcal{E}}

\def\cL{\mathcal{L}} 
\def\cM{\mathcal{M}} 
\def\cS{\mathcal{S}}
\def\cW{\mathcal{W}}

\def\cL{\mathcal{L}}
\def\cO{\mathcal{O}}

\def\sH{\mathscr{H}}

\def\sW{\mathscr{W}}

\def\iy{\infty}
\def\inv{^{-1}}

\def\pa{\partial}

\def\sgn{\textrm{sgn}}
\def\To{\Rightarrow}
\def\uga{\underline{\gamma}}

\newcommand{\one}{\mathbf{1}}
\newcommand{\la}{\langle}
\newcommand{\ra}{\rangle}
\newcommand{\nd}{\noindent}
\newcommand{\td}{\tilde}
\newcommand{\wtd}{\widetilde}

\newcommand{\hB}{\hfill$\Box$}
\newcommand{\hr}{\hookrightarrow}

\newcommand{\cob}[1]{{\color{blue}#1}} 
\newcommand{\cor}[1]{{\color{red}#1}}

\newcommand{\crr}{{\color{red}$\dagger$}}
\newcommand{\drr}{{\color{red}$\ddagger$}}

\newcommand{\rk}{{\bf Remark.}\ \ }

\newcommand{\Z}{\mathbb{Z}}
\newcommand{\R}{\mathbb{R}}
\newcommand{\C}{\mathbb{C}}
\newcommand{\N}{\mathbb{N}}

\def\sideremark#1{\ifvmode\leavevmode\fi\vadjust{\vbox to0pt{\vss
 \hbox to 0pt{\hskip\hsize\hskip1em
\vbox{\hsize2cm\tiny\raggedright\pretolerance10000
 \noindent #1\hfill}\hss}\vbox to8pt{\vfil}\vss}}}%

\newcommand{\edz}[1]{\sideremark{\color{blue}#1\color{black}}}
{\end{list}} 

\title[magnetic NLS equation]{Existence and non-existence of ground state solutions for  magnetic NLS} 
\author{Oleg Asipchuk} 
 \address[Oleg Asipchuk]{Department of Mathematics and Statistics\\
  Florida International University\\
Miami, FL 33199}
\email{aasip001@fiu.edu}
\author{Christopher Leonard} 
 \address[Christopher Leonard]{Department of Mathematics\\
 North Carolina State University\\
 Raleigh, NC 27695}
\email{cleonar@ncsu.edu}
\author{Shijun Zheng}
 \address[Shijun Zheng]{Department of Mathematical Sciences\\
 Georgia Southern University\\
 Statesboro, GA 30460-8093}
\email{szheng@GeorgiaSouthern.edu}

\begin{document} 
\maketitle

\begin{abstract}   
{We show the  existence and stability of ground state solutions (g.s.s.)
for $L^2$-critical magnetic nonlinear Schr\"odinger equations (mNLS) for a class of unbounded electromagnetic
potentials.  We then give non-existence result by constructing 
a sequence of vortex type functions in the setting of RNLS with 
an anisotropic harmonic potential.  
These generalize the corresponding results in \cite{ANS19stabi} and \cite{Dinh22mag3Drev}.  
The case of an isotropic harmonic potential for rotational NLS has been recently addressed in \cite{BHHZ23t}.  
Numerical results on the ground state profile near the threshold are also included. 
}
 \end{abstract}

\setcounter{tocdepth}{1} 

\tableofcontents    




\section{Introduction}\label{s:introAV} 

 The  magnetic nonlinear Schr\"odinger equation  (mNLS) in $\R^{1+d}$ reads
\begin{equation}\label{eqNls_AV}
	i\psi_t=-\frac{1}{2}(\nabla-iA)^2 \psi+ V \psi +N(x,\psi)
	\end{equation}
 with initial condition
\begin{equation*}
    \psi(0, x) = \psi_0(x),\quad x\in\R^d\,.
\end{equation*}
Here $(A,V)$ is the real-valued electromagnetic potential such that 
$A=(A_j)_1^d: \R^d\to \R^d$ is {\em sublinear} and $V$: $\R^d\to\R$ is {\em subquadratic}, 
that is,
$A\in C^\iy(\R^d,\R^d)$    and $V\in C^\iy(\R^d)$ satisfy  
$\pa^k A_j\in L^\iy$ for all $|k|\ge 1$  
 and $\pa^k V\in L^\iy$ for all $|k|\ge 2$ respectively. 
The two-form  $B=dA=(b_{jk})$, $b_{jk}=\pa_j A_k-\pa_k A_j$ denotes the 
magnetic field induced by the vectorial potential $A$.
The term  $N=N(x,\psi)$ has power-type nonlinearity. 
For the early physical background of the mNLS (\ref{eqNls_AV}),  
see  
\cite{AHS78mag,COMBES1978mNLS}  
and the references therein.    

We will primarily consider 
the existence problem of ground state solutions (g.s.s.) of the elliptic equation
\begin{equation}\label{eq:AV} 
 -\lam u=-\frac{1}{2} \nabla_A^2 u+ V(x)u +N(x,u)\,,
	\end{equation} 
where $\nabla_A:=\nabla -iA$ denotes the co-variant derivative.
The operator $\cL=H_{A,V}:=-\frac12\nabla_A^2+V$ is essentially selfadjoint in $L^2$ with the core $C_0^\iy(\R^d)$. 
This is the so-called Euler-Lagrange equation associated to Eq. (\ref{eqNls_AV}),
which can also be derived from the ansatz $\psi=e^{i\lam t} u(x)$, a standing wave solution of
the mNLS (\ref{eqNls_AV}).

The studies on the existence, stability and dynamics of ground state solutions for mNLS 
 have been motivated
by captivating interactions  and applications in  atomic and molecular physics, geometric optics,  
 spinor Bose-Einstein condensation (BEC),  quantum plasma and fluids.  
Recent decades have seen numerous theoretical  and computational methods developed in the field
\cite{AftalionAB05gvortex,AntDub2014gs,Aristov99mag-exact,
BaoCai2013num,BaoCai2017spinor,GaZh2013a,LewinNamRou2014mNLS,
SanSer2000magQ,Schmi2019magPauli,Zongo23Q_AB}, 
esp. around the stationary RNLS (\ref{eU:Om-ga}) and \eqref{RNLS_Omga} of type $(\Om,\ga)$. 
When $d=3$,  $A(x)=\Om\la-x_2,x_1,0\ra$ 
and $V_\ga(x)=\frac{\Om^2}{2}(x_1^2+ x_2^2)$, namely, $\ga=\Om (1,1,0)$,
the existence of g.s.s. and their stability were initially studied in  
 \cite{CazEs88,EsLion89}  for $p\in [1,5)$, 
 also cf.  \cite{Goncalves91instabiQ}.  
  Recent  studies have been conducted in 
\cite{ANS19stabi,Dinh22rmk,
BHHZ23t,LeoZ22n} for  $(\Om, \ga)$ of  anisotropic type, which have abundant physics motivations,
see e.g., \cite{Sowiski04rotrap,SmithDMS22rot-stats-ferm}.  
 
 Our first main theorem (Theorem \ref{T-Existence}) is to show the existence of g.s.s. for the 
 general magnetic NLS \eqref{eq:AV}   
 in the  focusing $L^2$-critical case  $p = 1 + \frac{4}{d}$.  
Throughout the article, we assume $A$ sublinear and $V$ subquadratic so that  
$ |A(x)|\le \al_0 (|V(x)|+1)^{1/2}$ for some $\al_0>0$. 

{We begin with stating the variational problem corresponding to (\ref{eq:AV}). 
The energy $E_{A,V}$ associated to (\ref{eq:AV})  is given by 
\begin{align} 
   E_{A,V}(u)=& \int \bar{u}\cL u\, dx + \int G(x,|u|^2) dx\,, \label{EAV:Q}
\end{align}
where  $G(x,|u|^2)=2\int_0^{|u|} N(x,s)ds$ is real-valued, 
viz., $N(x,u)=G'(x,|u|^2)u$ with 
$G'(x,z):=\pa_z G(x,z)$. 
Denote  $\Sigma=\sH^1$  the weighted Sobolev space endowed with  the norm 
\begin{equation}
\norm{u}_{\Sigma}=\left(\int |\nabla u|^2 + (1+|A|^2+|V|) |u|^2 \right)^{1/2}\,.
\end{equation}  
In view of  
(\ref{eAV:norm}) in Section \ref{s:prelimin}, we have norm equivalence on $\sH^1$
so that $\norm{\cdot}_\Sigma\approx \norm{\cdot}_{\Sigma_{A,V}}$.  

\begin{definition}
    Let
    \begin{equation}\label{eqMinimizationProblem}
        I_c=\inf\{E_{A,V}(u): u\in S_c\}
    \end{equation}
where $S_c=\left\{u\in\Sigma:\int |u|^{2}dx=c^2\right\}$, $c>0$.
    Then $u\in S_c$ is called a {\em ground state solution} (g.s.s) if $u$ is a solution to the following minimization problem
    \begin{equation}
        E_{A,V}(u)=I_c. \label{EAV:Ic}
    \end{equation}
Denote $\cO_c:=\{u\in S_c: E_{A,V} (u)=I_c\}$, that is, $\cO_c$ is the set of all minimizers of the problem 
(\ref{EAV:Ic}) in $S_c\subset \Sigma$. 
\end{definition}  
}      

\begin{hypo}\label{hy:G0-AV}  $V$ and $N$ satisfy the following conditions. 
	\begin{enumerate}
	\item[($V_0$):]\label{Ve:x2}  $V(x)\to +\iy$   as\ $|x|\to\iy$. 
	\item[($N_0$):]\label{it:G0} Let $N: \R^d\times \mathbb{C} \rightarrow \mathbb{C}$ be continuous such that
	   for some optimal constant $C_0>0$ and  $1< p\le 1+4/(d-2)$,
\begin{align}\label{G(xv)^p}
 \vert N(x,v) \vert \leq C_0 ( |v|+|v|^{p}). 
\end{align}
	\end{enumerate}
\end{hypo}
Then it is easy to see that the energy functional 
	$E_{A,V}$: $\Sigma \rightarrow \mathbb{R}$	is continuous. 	
Examples of nonlinearity of the above type include 
$\mu_1(x) |\psi|^{p-1}\psi+ \mu_2(x) |\psi|^{q-1}\psi$ for $p,q\in [1,\frac{d+2}{d-2}]$ and $\mu_1,\mu_2\in L^\iy$.

Let $Q=Q_0\in H^1(\R^d)$ be the unique positive, radial and decreasing solution of 
\begin{equation}\label{ground}
    -\frac{1}{2}\Delta u+u -|u|^{p-1}u=0\,, \quad 
\end{equation} %
cf. \cite{Kwong89,Merle96non-mini,wein1983sharp}.  


\begin{theorem}\label{T-Existence}
  Let $p=1+4/d$.    Let   $A$ be sublinear and $V$ subquadratic. 
     Let $Q_0$ be the positive radial solution of \eqref{ground}. 
Suppose   Hypothesis \ref{hy:G0-AV} is valid.  If $0<c< {C_0^{-\frac{d}{4}} } \norm{Q_0}_2$, then there exists a ground-state solution of Problem \eqref{EAV:Ic}.  That is, $\cO_c$ is non-empty in $\Sigma$.   
\end{theorem} 
{Theorem \ref{T-Existence} is sharp in the sense that  
there may exist no ground state solutions for (\ref{eq:AV}) if $c$ is above or equal to the threshold    
  $C_0^{-\frac{d}{4}} \norm{Q_0}_2$, 
see Remark \ref{re:nonE>c}} for relevant results in \cite{LeNamRou18a*,GuoLuoPeng23non}.  

{Note that the condition ($V_0$)  allows $V(x)$ negative if $|x|\lesssim 1$.  
Furthermore,  $(V_0)$ is much less  restrictive  than the  specific quadratic potential 
$V(x)=V_\ga(x)-\frac{|A|^2}{2}$ considered in \cite{EsLion89,ANS19stabi,BHHZ23t}.  
The growth of $V(x)$ suggests compact embedding property $\Sigma_{A,V}\hr L^r$
for $r\in [2,2^*]$, $2^*=2d/(d-2)$.  
The proof of Theorem \ref{T-Existence} is based on an adaption for the mNLS
of  a quite standard argument given in e.g., \cite{EsLion89,BHHZ23t,ZZhenZ19orbit}. 
The main ingredient is to establish the bound \eqref{Eu>Sigma} for $u$ in $S_c$
\begin{equation}
 E_{A,V}(u)\ge \td{\eps}\, \norm{u}_{\Sigma_{A,V}}^2- \td{C} \norm{u}_2^2
\end{equation}
provided $\norm{u}_2< {C_0^{-\frac{d}{4}} }\norm{Q_0}_2$.
As a corollary, 
we obtain that any minimizing sequence $(u_n)$ of (\ref{eqMinimizationProblem})  
 must be relatively compact in $\Sigma_{A,V}$, see Remark \ref{rem:mini-compact}.
}

\begin{definition} \label{def:orb-AV}
The set $\mathcal{O}_c$ is called {\em orbitally stable} if given any $\vphi\in \cO_c$, the following property holds:
For all $\varepsilon>0$, there is $\delta=\de(\veps)>0$ such that  
\begin{equation}\label{eOS:magV}
\norm{ \psi_0-\vphi }_{\Sigma}<\delta\Rightarrow \sup_t\inf_{\phi\in \cO_c} 
\norm{ \psi(t,\cdot)-\phi }_{\Sigma}<\varepsilon 
\end{equation}
where $\psi_0\mapsto \psi(t,\cdot)$ is the unique solution flow map of \eqref{eqNls_AV}.
\end{definition}  

\begin{hypo}\label{hy:V0-B0-G1} 
 \begin{enumerate}
	\item[($V_0$):]\label{Om-ga} $V(x)\to +\iy$ as $|x|\to \iy$.
\item[($B_0$):] 
The 2-form $B=dA=(b_{jk})$, $b_{jk}=\pa_k A_j-\pa_j A_k$ satisfies
\begin{equation*} 
 |\pa^k B(x)|\le C \la x\ra^{-1-\veps} \qquad\  \forall |k|\ge 1,
\end{equation*}
where $k=(k_1,\dots,k_d)\in\N_0^d$, $\N_0=\{0\}\cup \N$. 
\item[($N_1$):]\label{dG0} Let $N=N(x,z): \R^d\times \mathbb{C} \rightarrow \mathbb{C}$ be continuous and differentiable such that 
	   for some optimal constant $C_0>0$ and 	$1\le p\le 1+4/(d-2)$, 
\begin{align}\label{dG:p}
|\pa_z N(x,z)| \leq C_0( 1+|z |^{p-1}). 
\end{align}
\end{enumerate} 
\end{hypo}
Observe that Condition  ($N_1$) implies Condition $(N_0)$ with the same constant $C_0$,  
where both conditions $(N_i)=(N_{i,p})$,  $i=0,1$ are $p$-dependent.  
{Conditions ($B_0$) and ($N_1$) 
ensure  the global existence of the solution flow $\psi_0\mapsto \psi(t)$ of (\ref{eqNls_AV})
provided that $\norm{\psi_0}_2<C_0^{-d/4}\norm{Q_0}_2$.
 In addition,  Condition  ($V_0$) leads to the existence of ground state set $\mathcal{O}_c$. 
Consequently we can show that, if $c<C_0^{-d/4}\norm{Q_0}_2$,  
then $\mathcal{O}_c$, the set of g.s.s. of (\ref{eqNls_AV}) are orbitally stable 
by following a standard concentration compactness argument}.  

\begin{theorem}\label{t2:orb-Oc}  Let $p=1+4/d$.  
Let $A,V$ and $N$ verify the conditions in Hypothesis \ref{hy:V0-B0-G1}.  
Let  $0<c<C_0^{-\frac{d}{4}}\norm{Q_0}_2$. 
Then the set $\cO_c$ is orbitally stable for the magnetic NLS (\ref{eqNls_AV}).   
\end{theorem}   
If $c\ge \norm{Q_0}_2$, then Eq. (\ref{eqNls_AV}) may admit a finite time blow-up solution that means strong instability.  
 At the threshold $c=\Vert Q_0\Vert_2$, 
 the minimal mass blow-up solution problem was studied in \cite{BHZ19,BHHZ23t}. 
However, at the minimal mass level, the existence 
  of  ground states and  blow-up solutions for (\ref{eqNls_AV}) seem open for general magnetic NLS
 to our best knowledge. 
Note that the problem of orbital stability for a single ground state 
is  more difficult that may deserve further investigation.   

In the second part of this article,  let 
$\Om=(\Om_1,\dots,\Om_{[\frac{d}{2}]})\in \R^{[\frac{d}{2}]}$ and $\ga=(\ga_1,\dots,\ga_d)\in\R^d$. 
We consider the non-existence problem of (\ref{eq:AV}) 
for the rotational nonlinear Schr\"odinger equation  (RNLS) of type $(\Om,\ga)$ 
\begin{align}\label{eU:Om-ga}
&-\frac12\De u+V_\ga(x) u+ N(x,u) +L_A u=-\lam u,\qquad \norm{u}_2=1
\end{align}
where 
$L_A=i A\cdot\nabla$, $A=M x$,  $M=M_d(\Om)$  a skew-symmetric matrix, 
  \begin{align}\label{eV:ga-sgn} 
  V_{\gamma}(x)=\frac12\sum_j\eps_j\gamma_j^2 x_j^2
\end{align}
$\eps_j:=\sgn(\ga_j)$  and $N$ verifies the condition ($N_0$).  
The magnetic form of  (\ref{eU:Om-ga}) is given by \eqref{eU:QAV} 
namely, 
$-\frac12(\nabla-iA)^2 u+V_e(x) u+N(x,u)=-\lam u $,
  where     in particular, if $\ga_j>0$ for all $j$, then 
   \begin{align}\label{ecase:VgaOm} 
   V(x)=V_e(x)=\begin{cases}
   \frac12\sum_{j=1}^{[\frac{d}{2}]} (\ga_{2j-1}^2-\Om_j^2)x_{2j-1}^2
   +(\ga_{2j}^2-\Om_j^2)x_{2j}^2   &               \  d \ even\\
   \frac12\sum_{j=1}^{[\frac{d}{2}]} (\ga_{2j-1}^2-\Om_j^2)x_{2j-1}^2
   +(\ga_{2j}^2-\Om_j^2)x_{2j}^2 + \ga_d^2x_d^2 & \ d\ odd 
   \end{cases} 
  \end{align}
   per the time-dependent equation \eqref{L:AVe-cri} 
     in Subsection \ref{ss:OmgaRNLS}.
 The following theorem shows that if the condition ($V_0$) is not verified,
then there exists no minimizers for Problem (\ref{EAV:Ic}). 

\begin{theorem}\label{t3:non-E} Let $p\in [1, 1+4/(d-2))$.  
Let $A=M x$,  $M=M_d(\Om)$  an anti-symmetric matrix and let   
$V_\ga(x)=\frac12\sum_{j=1}^d \eps_j\ga_j^2x_j^2$ be given by (\ref{eV:ga-sgn}). 
Then there exists no ground state solutions for (\ref{eU:Om-ga}) under either of the following conditions: 
\begin{enumerate}
\item[(a)]  $|\Om_{k_0} | {>} \min(\ga_{2k_0-1},\ga_{2k_0})$ for some $1\le k_0\le [\frac{d}{2}]$;  
\item[(b)]  $\ga_{j_0}<0$ for some $1\le j_0\le d$. 
\end{enumerate}
 \end{theorem}   

In Section \ref{s:non-E:AV} we show the non-existence result in Theorem \ref{t3:non-E}
by constructing some counterexamples in the anisotropic region. 
In fact, we consider the class of quadratic potentials in the form (\ref{eV:ga-sgn}).
  The discussions are furnished with the following cases:  
\begin{enumerate}
\item[(i)] $0<\ga_1=\ga_2<|\Om|$, $d=2$ 
\item[(ii)] $0<\min(\ga_1,\ga_2)<|\Om|$,  $\ga_3\in\R_+$, $d=3$
\item[(iii)] $\ga_1,\ga_2>0$, $\ga_3<0$, $\Om\in \R$.  $d=3$.
\end{enumerate} 

The existence and stability of ground states as well as their instability for RNLS (\ref{eU:Om-ga}) 
were treated in the special case in $\R^3$ where 
 $|\Om|=\ga_1=\ga_2$ and $\ga_3=0$  in  
\cite{CazEs88,EsLion89}, \cite{Goncalves91instabiQ} and recently \cite{Dinh22mag3Drev} for $p\in [1,5)$ and $d=3$.   
Theorems \ref{T-Existence} and \ref{t2:orb-Oc} cover the case $|\Om|<\min(\ga_1,\ga_2)$, $\ga_3>0$. 
It would be of interest to study the remaining case $\ga_3=0$ and $|\Om|<\min(\ga_1,\ga_2)$. 

The organization of the remaining parts of this article is as follows.  
In Section \ref{s:prelimin} 
we briefly recall the $\sH^1$-subcritical
local in time existence result for (\ref{eqNls_AV}) as
well as some definitions and inequalities as required. 
In Section \ref{s:pf-lwpAV}  we provide a proof for the local existence theory of mNLS in $\Sigma=\sH^1$. 
In Section \ref{s:E-gss:AV}  we prove the existence and orbital stability of the set of magnetic ground states for (\ref{eq:AV}) in the $L^2$-critical 
regime  $p=1+4/d$.  
  In Section  \ref{s:non-E:AV} we show the non-existence of g.s.s for Problem (\ref{EAV:Ic}) 
  in some fast rotation cases.   
Finally, in Section \ref{s:numQAV} we present some numerical results to show the threshold dynamics on
the existence and blowup of the RNLS 
 (\ref{RNLS_Omga}) in 2D.  
We plan to continue the investigation on such existence problem for mNLS 
in the  $L^2$ supercritical regime in a sequel to this article.   

\section{Preliminaries}\label{s:prelimin}  

Define the weighted Sobolev space $\sH^{1,r}=\{u\in\cS'(\R^d): \nabla u\in L^r, (1+|A|^2+|V|)^{1/2} u\in L^r \}$ by the norm  
\begin{equation} \label{Sigma:r2}
\norm{u}_{\sH^{s,r}}=\left(\int \big\vert |\nabla|^s u\big\vert^r + \left(1+|A|^2+|V|\right)^{r/2} |u|^r \right)^{1/r}\,. 
\end{equation}
Then $\Sigma=\sH^1:=\sH^{1,2}$ is the energy space for mNLS (\ref{eqNls_AV}) so that  the functional 
$E=E_{A,V}$: $\Sigma\to \R$ is continuous. 
Denote  $\Sigma_{A,V}:=\{u\in L^2: \nabla_A u\in L^2,  |V|^{1/2} u\in L^2\}$  the magnetic Sobolev space 
for $\cL=H_{A,V}$  equipped with the norm 
\begin{equation} \label{eAV:norm}
\norm{u}_{\Sigma_{A,V}}=\left(\int \vert \nabla_A u\vert^2+ |V|\, |u|^2  + |u|^2 \right)^{1/2}\,. 
\end{equation} 
Evidently, we have  $\Sigma\subset \Sigma_{A,V}$.  For the other direction,  
 if $ |A(x)|\le \al_0 (|V(x)|+1)^{1/2}$ for  
  some $\al_0\in (0,\iy)$, 
    then the identification 
$\Sigma=\Sigma_{A,V}$ holds with equivalent norms $\norm{u}_{\Sigma} \approx\norm{u}_{\Sigma_{A,V}}$,
cf. Subsection \ref{magV:sobolev}. 


\subsection{Local existence for mNLS (\ref{eqNls_AV}) }\label{ss:lwp-mNLS}
 The local wellposedness (l.w.p.)  for (\ref{eqNls_AV}) were obtained 
 in  $\sH^1$ \cite{Bou91,Z12a} and \cite{HLeoZ21u}
for $A$ sublinear, $V$ subquadratic and $N(\psi)$ of power nonlinearity $\pm |\psi |^{p-1} \psi$,  
$p\in (1,1+\frac4{d-2})$.    
The $\sH^s$ subcritical result was also considered in \cite{Z12a} for $1\le p<1+4/(d-2s)$.
Here we state the local existence theory for mNLS (\ref{eqNls_AV}) with a general nonlinearity 
whose proof will be given in Section \ref{s:pf-lwpAV}.  

\begin{proposition}\label{p:lwp:AV} Let $p\in (1,1+\frac4{d-2}]$. 
Let $(A,V)$ be sublinear-subquadratic 
and $N$ satisfy the conditions  in Hypothesis \ref{hy:V0-B0-G1}. 
Let \mbox{$(q,r)=(\frac{4p+4}{d(p-1)} ,p+1)$} be an admissible pair given in Definition \ref{def:admissible}.  
Suppose $\psi_0\in \sH^1$. Then the following statements hold.    
\begin{enumerate}  
\item[(a)] If $p\in (1,1+\frac4{d-2})$, then there exists  a unique solution  $\psi$ of (\ref{eqNls_AV}) 
 in $C(I_*,\sH^1) \cap L^q(I_*, \sH^{1,r})$ such that     
 $I_*:=(-T_-,T_+)$, $T_\pm>0$  is the maximal lifespan of the  solution.    
\item[(b)]  If $p=1+4/(d-2)$, $d\ge 3$,  then there exist $T_\pm=T_\pm(\psi_0)$ 
such that $\psi(t)\in C( I_*,\sH^1)\cap L^q( I_*, \sH^{1,r})$.   
\item[(c)]  There exists a global solution $\psi$ of (\ref{eqNls_AV}) in $C(\R,\sH^1)\cap L^q_{loc}(\R, \sH^{1,r})$ 
provided one of the following conditions holds:
\begin{enumerate}
\item[(i)]  $p\in (1,1+\frac4d)$;
\item[(ii)]   $p\in (1,1+\frac4{d-2})$ and $G(x,|z|^2)\ge 0$; 
\item[(iii)]   $p=1+4/(d-2)$, $d\ge 3$ and  $\Vert \psi_0\Vert_{\sH^1}<\veps=\veps(d,G)$ is sufficiently small.  
\end{enumerate}
\item[(d)] On the lifespan interval $I_*=(-T_{-},T_{+})$, we have the following conservation laws: 
\begin{align}  
(mass)\quad M (\psi)=&\Vert \psi(t)\Vert_2^2=\Vert \psi_0\Vert_2^2\label{mass-AV}  \quad \\ 
(energy)\quad E_{A,V}(\psi)=&E_{A,V}(\psi_0)  \label{Eu:AV}    \\
=&\int \left(\frac12| \nabla_A \psi |^2+V |\psi|^2 + G(x,|\psi|^2) \right)\,. \notag 
\end{align} 
\item[(e)] \textup{[blow-up alternative]} If $T_+<\iy$ is finite \mbox{(respectively, $T_-<\iy$)}, 
then $\psi$ blows up finite time in the sense that 
\begin{align*}
& \lim_{t\to T_+} \norm{\nabla \psi}_2=\lim_{t\to T_+} \norm{\psi}_\iy=\iy\,.\\
&(resp.   \quad  \lim_{t\to T_-} \norm{\nabla \psi}_2=\lim_{t\to T_-} \norm{\psi}_\iy=-\iy).
\end{align*}
\end{enumerate}  
\end{proposition}   
\begin{remark}\label{re:S(I)}  In the statements (a)-(c), all versions of $L^q( \td{I}, \sH^{1,r})$,
$\td{I}$ an interval in $\R$,
can be substituted with the Strichartz space $S_{\td{I}}$ as defined in (\ref{norm:S_I}). 
This can be proven following an evident modification of the proof of Proposition \ref{p:lwp:AV} in Section \ref{s:pf-lwpAV}. 
In fact, for the modification we only need to use the space $\td{X}_I:=S_I $ in place of $X_I= L^\iy(I, \sH^1)\cap L^q(I, \sH^{1,r}) $ 
in Subsection \ref{ss:proof-lwp:AVG}.  
\end{remark}  

\bigskip 
\subsection{Rotational NLS} \label{ss:OmgaRNLS}  The RNLS is a special form of the magnetic NLS:
\begin{align}    
&i\psi_t=-\frac12\De \psi+V_\ga(x) \psi+ N(\psi) +L_A \psi\  \label{RNLS_Omga}\\     
& \psi(0,x)=\psi_0\in \sH^1\, \nonumber    
\end{align}          
where $V_\ga$ and $L_A=i (Mx)\cdot\nabla $ are defined in (\ref{eU:Om-ga}); 
$N(\psi)=G'(x,|\psi|^2)\psi$ is given via (\ref{EAV:Q}).   
 It can be written in the magnetic form  
 \begin{align}\label{L:AVe-cri}  
i\psi_t=\mathcal{L}\psi+ N(\psi), \quad \psi(0,x)=\psi_0\in \sH^1
\end{align}  
where 
   \begin{equation*} 
\mathcal{L}=-\frac{1}{2} \sum_1^d (\frac{\pa}{\pa x_j}-iA_j(x))^2u+V\,  
\end{equation*}
such that $V(x)=V_\ga(x)-\frac{|A|^2}{2}$.

If $A$ is in the rotational class, i.e., $A=Mx$, $M^T=-M$, then
 one can always write $M=O^T \td{M}O$, where $O$ is an $d$ by $d$ orthogonal matrix in $SO(d)$. 
  Here $\td{M}=\td{M}_d(\Om)$ takes the following form: If $d=2$
\begin{align*}
\td{M}_2=\td{\Om}\sigma,\quad \sigma=\begin{pmatrix}
 0&-1 \\
 1 &0
\end{pmatrix};
\end{align*}
if  $d=3$, 
$\displaystyle \td{M}_3=\begin{pmatrix} 
  \td{\Om}\sigma& {} \\  
{}   &0
\end{pmatrix} $, 
where $\td{\Omega} \in\R$. In general,  
for even dimensions, we define $\td{M}_d=\td{M}_d(\Om)$ as a matrix with $\Om_j\sigma$ on the diagonal, 
where $\Om:=(\Om_1,\dots,\Om_{[\frac{d}{2}]} )$, $\Om_j\in\R$, $j=1,\dots, \frac{d}{2}$, i.e.,
\begin{align}\label{eqMeven}
\td{M}_d=\begin{pmatrix} 
 \Om_1\sigma& 0 &\dots&0\\
 0 & \Om_2\sigma &\dots&0 \\ 
\vdots& \vdots &\ddots& \vdots\\
0&0&\dots &\Om_{\frac{d}{2}} \sigma
\end{pmatrix}.
\end{align}
For odd dimensions, we define $\td{M}_d=\td{M}_d(\Om)$ as a matrix with 
$\Om_j\sigma$  on the diagonal,  $j=1,\dots,\frac{d-1}{2}=[\frac{d}{2}]$, with zero on the last entry 
i.e.,
\begin{align}\label{eqModd}
\td{M}_d=\begin{pmatrix} 
 \Om_1\sigma & 0 &\dots&0&0\\
 0 & \Om_2\sigma &\dots&0&0 \\ 
 {\vdots}& \mbox{\vdots} &\ddots& {\vdots} & {\vdots}\\
0& 0 &\dots& \Om_{\frac{d-1}{2}} \sigma&0\\
0&0&\dots &0& 0  
\end{pmatrix}.
\end{align}  
Via an obvious transform $\psi(t,x)\mapsto \td{\psi}(t,O x)$ one can 
convert any RNLS (\ref{RNLS_Omga}) into a NLS ($\widetilde{\text{RNLS}}$)
in the pattern  so that $\td{A}(x)=\td{M}x$ and $\wtd{V_\ga}(x)= V_{\ga}(Ox)$ is some quadratic function.   
Then a change of time-dependent coordinates $\td{\psi}(t,x)\mapsto \td{\psi}(t,e^{t\td{M}} x)$   
will further convert the $\widetilde{\text{RNLS}}$ to a rotation free NLS with a time-dependent quadratic 
potential $\td{V}(t,x)$. 
Throughout this article, for simplicity 
we shall assume that the RNLS (\ref{RNLS_Omga})  and (\ref{eU:Om-ga}) 
admit $M=\td{M}_d(\Om)$ and $\widetilde{V_\ga}(x)=V_{\td{\ga}}(x)$  
for some $\Om\in \R^{[\frac{d}{2}]}$ and $\td{\ga}\in \R^d$. 


\subsection{Magnetic Gagliardo-Nirenberg (GN) inequality} \label{ss:mGN}
For the proof of Theorem \ref{T-Existence} concerning 
the existence of magnetic ground state solutions of (\ref{eq:AV}),  
that is, global minimizers of Problem (\ref{EAV:Ic}), we will need the following magnetic inequalities. 
\begin{lemma}[Diamagnetic inequality]\label{l:diamag}
 Let $A\in L^2_{loc}(\R^d,\R^d)$. Then 
\begin{equation}
 \norm{\nabla \vert u\vert}_2\le \norm{(\nabla-iA) u}_2\,. \label{diamag:A}
\end{equation}
\end{lemma}

The next lemma is a modification of the classical sharp GN inequality \eqref{eqGN},  
which can be applied to solve the minimization problem under the threshold $\norm{Q_0}_2$. 
\begin{lemma} \label{l:mGN_sharp} 
Let $p=1+4/d$. Then  there is a sharp constant  
    \begin{equation*}
        C_{GN} = \frac{d+2}{2d ||Q||_{2}^{\frac{4}{d}}}
    \end{equation*}
    such that for all $u\in\Sigma_{A,V}$
    \begin{equation}\label{eqGNA}
    ||u||_{p+1}^{p+1}\leq C_{GN}||\nabla_A u||_{2}^2||u||_{2}^\frac{4}{d}.
\end{equation} 
\end{lemma} 
{ In the lemma, the equality holds only if $A$ is a conservative field, i.e., 
 $A=\nabla \phi$ for some $\phi\in C^1$;
or equivalently, $B=\nabla \wedge A=0$}. 
Lemmas \ref{l:diamag}-\ref{l:mGN_sharp} can be found  in \cite{EsLion89,ArioSzu03mag,Dinh22mag3Drev,BHHZ23t}.  
The latter  is an immediate corollary of  (\ref{diamag:A}) and the well-known GN inequality: 
Let $p\in (1,1+\frac4{d-2})$, then   
 \[ || f ||_{p+1}\leq C ||\nabla f ||_{q}^\theta || f ||_{r}^{1-\theta}  \]
 where $\frac{1}{p+1}=(\frac{1}{q}-\frac{1}{d})\theta+\frac{1-\theta}{r} $.  
In particular, if $p=1+4/d$, then  for all $u\in H^1$, it holds that    
\begin{equation}\label{eqGN}
    \norm{u}_{p+1}^{p+1}\leq\frac{d+2}{2d\norm{Q}_{2}^{\frac{4}{d}}} \norm{\nabla u}_{2}^2\norm{u}_2^\frac{4}{d}\,,
\end{equation} 
where $Q$ is the ground state solution to \eqref{ground}. 

\subsection{Energy for a general electromagnetic field}\label{ss:E_AV}

For general $(A,V_e)$, the  linear hamiltonian 
$\cL=H_{A,V_e}:=-\frac12(\nabla-iA)^2+V_e$,
where the notation $V_e=V_e(x)$ means an effective (electric) potential
that is identical with $V(x)$ in Eq. (\ref{eqNls_AV}).  
As is known \cite{GaZh2013a,LewinNamRou2014mNLS}, the mNLS is generated by the hamiltonian 
$\mathcal{E}(\psi)$ in the symplectic form $\psi_t= J\cE'(\psi)$ for $J^2=-I$ 
and $\cE$: $\psi\mapsto \cE(\psi)=E_{A,V_e}(\psi)$, i.e.,   
\begin{equation}\label{eAV:u(t)}
	i\psi_t=-\frac{1}{2}(\nabla-iA)^2 \psi+ V_e(x)\psi +N(\psi) \,.
	\end{equation}
The associated energy on $\sH^1=\Sigma_{A,V_e}$ 
is given by 
\begin{equation}\label{EAVe:mu}
\begin{aligned} 
   \cE(\psi)=& 
    \frac12\int  |\nabla_A \psi |^2 +\int V_e(x) |\psi|^2 +  \int G(|\psi |^2)  \\
=&  \frac12\int  |\nabla \psi |^2 +\int (V_e(x)+\frac{|A|^2}{2} )  |\psi|^2 \\
+& i \int \bar{\psi} (A\cdot \nabla \psi)+ \frac{i}{2}\int (\dive A) |\psi|^2 +\int G(|\psi |^2)\,. %
\end{aligned} 
\end{equation} 
In the setting of  RNLS (\ref{RNLS_Omga}), if   $d=3$, 
$A(x)=\Om\la -x_2,x_1,0\ra$ and $V_\ga(x)=\frac12\sum_j \ga_j^2x_j^2$, $\ga_j>0$, then we have 
in view of (\ref{EAVe:mu})  
\begin{align*} 
 E_{\Omega,\ga}(\psi)&=\frac{1}{2}\norm{\nabla \psi }_2^2 + \int V_\ga(x)|\psi |^2 
        + \int G(|\psi|^2) - i\Omega\int\bar{\psi} x^\perp\nabla \psi\,. 
\end{align*} 
It has an expansion  
\begin{align}
&E_{\Om,\ga}(\psi)= \frac12\int |\nabla \psi|^2 +\frac12\int\left(\ga_1^2x_1^2+\ga_2^2x_2^2 +\ga_3^2 x_3^2\right) 
|\psi|^2\notag\\
        +&\int G(|\psi|^2) - i\Om \int \bar{\psi} x^\perp \cdot \nabla \psi \label{E3:OmV} \\
      =& \frac12\int |(\nabla-iA) \psi |^2 
      +\frac12\int\left( (\ga_1^2-\Om^2)x_1^2+(\ga_2^2-\Om^2)x_2^2 +\ga_3^2 x_3^2\right) |\psi |^2\notag\\
        +& \int G(|\psi|^2)\,,  \notag 
    \end{align}
where  $N(\psi)=G'(|\psi|^2)\psi$ 
and $A= -\Om x^\perp$. In this article, we denote  $x^\perp=\la x_2,-x_1,0\ra$ if $d=3$;  $\la x_2,-x_1\ra$ if $d=2$.  

If $(A,V_e)$ is a sublinear-subquadratic pair, then minimization of the energy (\ref{EAVe:mu}) gives rise to the 
associated Euler-Lagrange equation 
 \begin{equation} \label{eU:QAV} 
-\lam u=-\frac{1}{2}(\nabla-iA)^2 u+ V_e(x)u + G'(|u|^2)u.
\end{equation}  
In  Subsection \ref{ss:EL-QAV} 
we shall see that the minimizers  of (\ref{EAVe:mu}) 
satisfy (\ref{eU:QAV}).  

The hamiltonian $\cE(\psi)$ has a nonlinear potential component 
that induces the nonlinear term $N(\psi)=G'(|\psi |^2)\psi$ in (\ref{eAV:u(t)}). 
In geometric optics, the cubic NLS is a  paraxial approximation to the light wave diffraction and 
dispersion when the wavelength is much smaller than the input beam radius.  
For the special form $N=N(p):=\mu |\psi|^{p-1}\psi$, 
the exponent $p=3$ arises in the statistical interpretation, that is, 
the density $\wtd{\sW}=|\psi|^2$ plays the physical role of a waveguide, whose expectation 
is given by $\la \wtd{\sW} \psi \mid \psi\ra$ in the hamiltonian. 
In this manner the nonlinearity in the dimensionless model near the interface 
 describes the self-interactions between particles.
One can derive the cubic NLS from the cubic  Helmholtz equation  \cite{BaFibTsyn08backscattQ}. 
Other general exponents of $p$ have appeared in a variety of physical settings, e.g., the celebrated Lee-Huang-Yang
correction on the effect of quantum fluctuation on the ground state energy of a bosonic gas \cite{LeeHuangYang1957}.
See also \cite{BERGE1998nonlinearity,KohlBiswas08non-kerr} and the references therein.   

\section{Local wellposedness for mNLS (\ref{eqNls_AV}) }\label{s:pf-lwpAV}  
The proof of Proposition \ref{p:lwp:AV} can be carried out by following suitable modification of a contraction mapping argument given in \cite{Bou91,Z12a},
where the special term $N(\psi)=\pm |\psi|^{p-1} \psi$ is involved, in which case the solution $\psi$ exists in 
$C(I,\sH^1) \cap L^q(I, \sH^{1,r})$ for some time interval $I$ containing $0$,
where  $r=p+1$ and  $q=\frac{4p+4}{d(p-1)}$.  
Here we briefly present the outline of the proof for the general nonlinearity $N(\psi)=G'(x,|\psi|^2)\psi$ 
satisfying \eqref{dG:p}. 
We show that there exists a unique solution in the Strichartz space $\td{X}_I$ as defined in Subsection \ref{ss:strichartz-X_I}. 
Recall the mNLS from Eq. \eqref{eqNls_AV}
 \begin{equation}\label{mNLS:AV}
\begin{cases}
i\partial_t \psi= -\frac12\Delta_A \psi+V \psi +N(\psi) \\
\psi(0,x)=\psi_0(x)\in \mathscr{H}^1.
\end{cases}
\end{equation}
Let $\cL=H_{A,V}$ and $U(t)=e^{-it\cL}$ be the linear propagator. 
Note that a weak solution of (\ref{mNLS:AV}) in $\td{X}_{{I}}$, a suitable  subset in  $C({I},\sH^1)$, 
${I}:=(-T,T)$, $T>0$
is equivalent to the solution of the integral equation
\begin{align}  
& \psi=U(t)\psi_0- i\int_0^t U(t-s) N(\psi(s,\cdot) ds \label{u=duhamel}
\end{align} 
 in $\td{X}_{{I}}$. 

\subsection{Dispersive estimates for $U(t)=e^{-itH_{A,V}}$} \label{ss:disp-U(t)}
\begin{definition}\label{def:admissible} The pair $(q,r)$ is called {\em admissible} provided that
  $q\in [2,\iy]$, $r\in [2,2^*]$, $2^*=2d/(d-2)$, 
  $(q,r,d)\ne (2,\iy,2)$   satisfying 
$\frac{1}{q}=\frac{d}{2}(\frac12-\frac{1}{r})$.  
\end{definition} 
From \cite{Ya91} the $L^1\to L^\iy$ dispersive estimate hold for some $\de_0>0$ 
\begin{equation}\label{eU(t):disp}
 \Vert U(t-s)f\Vert_\iy\le C |t-s|^{-d/2}\Vert f\Vert_1\,, \quad\ \forall |t-s|<\de_0.
\end{equation}
Define the modified propagator $\mathcal{U}(t)=U(t)\one_{ \{|t|<\de_0\}}$. By applying 
the $TT^*$ argument for $\mathcal{U}(t)$ 
in \cite{KT98} and the norm-characterization of $\sH^s$ in \cite{Z12a},   
more generally the commutator relations between $\pa_j$, $x_j$ and the oscillatory integral  for $U(t-s)$ 
described in \cite{HLeoZ21u},  via interpolation  for $\sH^k$, $k\in \Z_+$ 
 we obtain the $\sH^s$-Strichartz type estimates for $\cL$: For all $s\in \R_+$   
\begin{align} 
&\Vert U(t) f \Vert_{L^q(I_0, \sH^{s})}\le  C _{s,d,q,r} \Vert f \Vert_{\sH^s}  \label{qr:stri-L}\\ 
&\Vert \int_{I_0} U(t-s)F(s,\cdot) ds\Vert_{L^q(I_0,  \sH^{s,r} )}\leq C_{s,d,q,r,\td{q},\td{r}} 
  \Vert F \Vert_{L^{\td{q}'}(I_0, \sH^{s,\td{r}'} )} \,,\label{inhomo-stri-L}
\end{align}  
where $I_0=(-\de_0,\de_0)$, $(q,r)$, $(\td{q},\td{r})$ are any admissible pairs, and $q'$ denotes the H\"older conjugate exponent of $q$. 

\subsection{Strichartz spaces for $H_{A,V}$} \label{ss:strichartz-X_I}
 Let  $I=(-T,T)$ for $T>0$. Introduce the Strichartz  space   
 $S_I:=
  \cap_{(q,r)\ admissible} L^q (I, \sH^{1,r} ) \,  $ 
   endowed with the norm 
\begin{align} \label{norm:S_I}
\norm{u}_{S_I}=& \sup_{(q,r)\ admissible}\Vert u(t,x)\Vert_{L^q(I,\sH^{1,r})}\,.
\end{align}  
The dual space $N_I=\cup_{(q,r)\ admissible} L^{q'}(I,\sH^{1, r'})
=L^1(I, \sH^{1,2}  )+L^2(I, \sH^{1, \frac{2d}{d+2}}  )$ in 
$ \cS'(\R^{1+d}) $, $d\ge 3$ is 
endowed with the norm
\begin{align*} 
 &\norm{u}_{N_I} 
= \inf_{(q,r)\ admissible}\Vert u(t,x)\Vert_{L^{q'}(I,\sH^{1,{r'}} )} \,.
\end{align*}  
If $d=1,2$, one simply write $N_I=\cup_{(q,r)\ admissible} L^{q'}(I,\sH^{1, r'})$.  
Then the Strichartz space $(\td{X}_I, \norm{\cdot}_{S_I})$ where
\begin{equation}\label{XT:qr}
\td{X}_I=S_I\cap C(I,\sH^1) 
\end{equation} 
is a Banach space in  $\mathcal{S}'(\R^{1+d})$.  

By Duhamel principle for (\ref{mNLS:AV})-(\ref{u=duhamel}) consider the map  
\begin{align}  
& \Phi(u):=U(t)\psi_0- i\int_0^t U(t-s) N(x,u) ds .\label{Phi(u):duhamel} 
\end{align}   
Let $E_I=\{u\in \td{X}_I: \Vert u\Vert_{S_I} 
\le 2C \Vert \psi_0\Vert_{\sH^1} \}$, where $C$ is a constant dependent  on 
the constants in the priori estimates (\ref{qr:stri-L})-(\ref{inhomo-stri-L}) only.  
{Note that the two inequalities 
hold with a constant $\td{C}$ that is independent  of all admissible pairs $(q,r)$, $d\ne 2$. 
We will show that $\Phi$ is a contraction mapping on the closed set $E_I$ with suitable topology 
as given in Subsection \ref{ss:proof-lwp:AVG}. 

\begin{lemma}\label{l:deF_qr-Z} Let $1\le p\le 1+4/(d-2)$ and $I=I_T:=(-T,T)$, $T>0$. 
Let $1/\sigma= 1-{2}/q$, $r=p+1$, $q= \frac{4p+4}{d(p-1)}$. 
Suppose  $F(x,z)\in C(\R^d,\C)$ satisfies $F(x,0)=0$ and for all $x, z$ 
\begin{align}\label{e:derF_p}  
& | \pa_z F (x,z)|\le  C |z|^{p-1}\,. \quad 
\end{align} 
Then we have for $u,v\in L^\iy(I ,L^r)\cap L^q(I ,\sH^{1,r})$ 
\begin{align*}
&  \Vert \la x\ra F(x,u )\Vert_{L^{q'}(I,L^{r'})} 
\le C T^{1/\sigma} \Vert u\Vert^{p-1}_{\iy,r} \Vert   \la x\ra u \Vert_{q,r}  \,,\\
& \Vert \nabla_z F(x,u )\Vert_{L^{q'}(I,L^{r'})} 
\le C T^{1/\sigma} \Vert u\Vert^{p-1}_{\iy,r} \Vert \nabla u \Vert_{q,r}  \,,\notag\\
&  \Vert  F(x,u )-F(x,v)\Vert_{L^{q'}(I,L^{r'})} 
\le C T^{1/\sigma} (\Vert u\Vert^{p-1}_{\iy,r}+\Vert v\Vert_{\iy,r}^{p-1}) \Vert u-v \Vert_{q,r} \, \notag
\end{align*}  
where $\norm{ u}_{q,r}:=\norm{ u}_{L^q(I,L^{r} ) }$.   
\end{lemma}  

\begin{proof} First note that  \begin{align}\label{e:F_p-ineq}
& | F(x,u)-F(x,v)|  \le C (|u|^{p-1}+ |v|^{p-1})|u-v| ,
\end{align}
which follows from \eqref{e:derF_p} and 
\begin{align*}
F(u)-F(v)= \int_0^1 F'((1-\theta)u+\theta v) d\theta \cdot (u-v) .
\end{align*}
In view of \eqref{e:derF_p} and \eqref{e:F_p-ineq}, applying H\"older inequality with $\frac{1}{r'}=\frac{p-1}{r}+\frac{1}{r}$, $\frac{1}{q'}=\frac{1}{\sigma}+\frac{p-1}{\iy}+\frac{1}{q}$, we obtain
\begin{align*}
& \Vert \la x\ra F(u )\Vert_{L^{q'}(I,L^{r'})} \le CT^{1/\sigma} \Vert  u\Vert^{p-1}_{\iy,r}\Vert \la x\ra u\Vert_{q,r}\,, \\
&\Vert \nabla F(u )\Vert_{L^{q'}(I,L^{r'})} \le CT^{1/\sigma} \Vert  u\Vert^{p-1}_{\iy,r}\Vert \nabla u\Vert_{q,r} \, ,
\end{align*}
and 
\begin{align*}
& \Vert F(u )-F(v)\Vert_{L^{q'}(I,L^{r'})}\le CT^{1/\sigma} (\Vert  u\Vert_{\iy,r}+\Vert v\Vert_{\iy,r})^{p-1} \Vert  u-v\Vert_{q,r} \,.
\end{align*}
\end{proof} 
\begin{remark}   Note that in the proof of part (a) in Proposition \ref{p:lwp:AV}, 
we would not need to make an estimate  of 
$ \Vert \nabla (F(u )-F(v))\Vert_{L^{q'}(I,L^{r'})} $, this can save a loss of derivative for $\psi_0\in \sH^1$. 
\end{remark}

\begin{remark}  
Concerning the exponent for $T^{1/\sigma}$, we note that  
if $ F(u)= |u|^{p-1}u$,  then $1/\sigma>0\iff 1<p<1+4/(d-2)$;  
if $F(u)=u$, then   the the exponent 
can be chosen to be greater than $1/\sigma$.  
So, it is essential to deal with the nonlinearity of the type $F(u)=O( |u|^p)$.  
\end{remark}

\subsection{Proof of Proposition \ref{p:lwp:AV} }\label{ss:proof-lwp:AVG} 

We mainly follow the same line of proof in \cite{Z12a} with simplifications and modifications.  
However, Proposition \ref{p:lwp:AV}  shows additional properties for the solution of (\ref{eqNls_AV}) for  $(A,V)$ and nonlinearity $N(x,\psi)=G'(x,|\psi|^2) \psi$
that are less restrictive than \cite{Z12a}.

\begin{proof}[Proof of Proposition \ref{p:lwp:AV}]  
For simplicity we solve the integral equation \mbox{$\Phi(u)=u$,} 
in the space $X_I:=L^\iy(I, \sH^1)\cap L^q(I, \sH^{1,r}) $,   
 where 
$\Phi$ is defined as in \eqref{Phi(u):duhamel}.  
An evident modification of the proof yields the results for the solution space $\td{X}_I$ defined in (\ref{XT:qr}).  

Let $ p\in[1, 1+\frac4{d-2}) $. 
Introduce $Z=L^\iy(I, L^2)\cap L^q(I, L^r) $ and
 $X_I$ as above. 
 Define $\Vert u\Vert_{ X_I}:=\Vert u\Vert_{L^\iy(I, \sH^1)}+ \Vert u\Vert_{L^q(I, \sH^{1,r})}$.
 Denote $E_{2\ga}=\{u\in X_I: \Vert u\Vert_{X_I} \le 2\ga  \}$,
 $\ga>0 $  to be chosen in a moment. 
Then 
{$(E_{2\ga},\rho)$ is a closed set in $X_I$ endowed with the ({weaker}) metric $\rho(u,v)=\Vert u-v\Vert_{Z}$}. 
 
To see this, let $M=2\ga$ and let $\{ u_n\}\subset E_M=\{ u\in X_I:\Vert u_n\Vert_{X_I}\le M \}\subset X_T\subset Z$ be a Cauchy sequence:  
\begin{align*}
&\Vert u_n-u_m\Vert_{q,r }\to 0 \quad as\;\,n,m\to \iy.
\end{align*}
Then there exists $\td{u}\in Z$ so that  
\begin{align}\label{e:u_n-u_qr}
\Vert \Vert u_n(t, x)-\td{u}(t,x)\Vert_{L^r }\Vert_{L^q(I)}\to 0 . 
\end{align}
We want to show that $\td{u}$ belongs to $E_M$.  Indeed,
\eqref{e:u_n-u_qr} suggests that, without loss of generality, we may assume $\lim_n  u_n(t, x)=\td{u}(t,x)$ a.e. in $t,x$.  

Next we show $\norm{\nabla_x \td{u} (t,x)}_{q,r}$ is finite and less than $M=2\ga$.   
By the weakly compactness of $E_M$, a bounded ball in $X_I$,  
 there exists ${w}(t,x)$ in $E_M\subset X_I$ 
\begin{equation}
w-\lim_n u_n( \cdot, \cdot)={w}\quad in \  \text{weak-topology of $X_I$}      
\end{equation}  
That is, for all $\phi(t,x)$ in $C_0^\iy( I\times \R^d) $ 
\begin{align*}
& \lim_n \la u_n, \phi\ra= \la \td{w}, \phi\ra. 
\end{align*}  

Now, in view of $  \la u_n, \phi\ra= \int u_n \phi  dx dt$
and (\ref{e:u_n-u_qr}) we know that (up to subsequences)
\begin{align}\label{Un:to-u_qr}
\lim_n u_n(t, x)=\td{u}(t,x)=w(t,x) \qquad a.e. \  (t,x) .       
\end{align}
Therefore, $\td{u}$ is in $E_M=E_{2\ga}$.  \hB

This is due to the  fact that any bounded set in a reflective $B$ space is weakly pre-compact;  
here we take $B=L^\iy(I, \sH^{1,2})\cap L^q(I,\sH^{1,r})$. Note that  $E_{M}$ is a bounded ball in $X_I$. 

\begin{remark}\label{r:w-topology-X_I}
To show the contraction property for $\Phi$,  it is crucial to work on the space $(X_I, \norm{\cdot}_Z)$ 
with a weaker topology than $(X_I, \norm{\cdot}_{X_I})$. 
The point is that any bounded sequence in a closed set $B$ in a (reflexive) topological space $X$
must be weakly pre-compact, hence bounded.  
Examples include a bounded set in a bounded domain $\cD$:
\begin{enumerate}
\item[(i)] $H^1(\cD)$ with the metric  of $L^2(\cD)$;  
\item[(ii)] $C^s(\cD)$, $s>0$ with metric of $C(\cD)$;
\item[(iii)] $L^\iy(\cD)$ with metric of $L^1(\cD)$.        
\end{enumerate}
In contrast, let $X_\Om=L^1(\Om)$.  Denote by $\cM$ the space of bounded measures on an open domain in $\Om\subset \R^d$, 
which is defined as the dual of $C_0(\Om)$,  
 that is,  $\cM$ consists of all continuous linear functionals of $C_0(\Om)$.
Then   $(X_\Om, \norm{\cdot}_{L^1(\Om)})$ and the closure of $(X_\Om, \norm{\cdot}_{\cM})$ are not the same. 
In fact, one can check this statement by taking a sequence $(f_n)$ in $L^1(\Om)$ that is an approximation to the 
identity with $\int |f_n| =1$. 
\end{remark} 


We are ready to prove that the mapping $\Phi$ in (\ref{Phi(u):duhamel}) is a contraction  on ($E_{2\ga}, \rho)$. 

($a_1$) We show that $\Phi$ is stable: $\Phi(E_{2\ga})\subset E_{2\ga}$. Let $u\in E_{2\ga}$, i.e., $\Vert u\Vert_{X_I}\le 2\ga$. 
Using (\ref{qr:stri-L})-(\ref{inhomo-stri-L})  and  Lemma \ref{l:deF_qr-Z} we have 
\begin{align*} 
& \Vert \Phi(u) \Vert_{X_I} \le C (\Vert \psi_0\Vert_{\sH^1}+\Vert  N(x,u) \Vert_{L^{q'}(I, \sH^{1,r'})}  )\\
\le&\ga + C' T^{1/\sigma} \Vert \psi\Vert^{p-1}_{\iy,r} \Vert  u \Vert_{L^q(I, \sH^{1,r})}\\
\le&\ga+\ga=2\ga,
\end{align*}  
where we use $\sH^{1}\hr L^r$ 
and choose $\ga= C\Vert \psi_0\Vert_{\sH^1}$ and 
\begin{align*}
&  T \le T_0(\norm{\psi_0}_{\sH^1} ):=\left(\frac{1}{2C' (2C\Vert \psi_0\Vert_{\sH^1} )^{p-1} } \right)^{\sigma} \,,
\end{align*}  
with $1/\sigma=1- 2/q>0$. 
 Note that  $\sigma >0$ if and only if  $ p< 1+4/(d-2)$.  

($a_2$) Show that $\Phi$ is a contraction. 
 Using \eqref{qr:stri-L}-\eqref{inhomo-stri-L} and Lemma \ref{l:deF_qr-Z} we obtain
\begin{align} 
& \Vert \Phi(u)-\Phi(v) \Vert_{Z} 
\le C\Vert F(u)-F(v) \Vert_{L^{q'}(I, L^{r'})}\notag\\
\le&C^{\prime} T^{1/\sigma} (\Vert u\Vert^{p-1}_{\iy,r}+\Vert v\Vert_{\iy,r}^{p-1}) \Vert u-v \Vert_{q,r} \notag\\
\le&C^{\prime\prime} T^{1/\sigma} (\Vert u\Vert^{p-1}_{X_I} +\Vert v\Vert^{p-1}_{X_I})\Vert  u-v \Vert_{Z}
\le \frac12 \Vert  u-v \Vert_{Z}\,,  \label{eZ:u-v}
\end{align} where  
we require 
\begin{align*}
& T \le T_1(\norm{\psi_0}_{\sH^1}) :=\left(\frac{1}{4C^{\prime\prime} (2C\Vert \psi_0\Vert_{\sH^1})^{p-1}} \right)^\sigma . 
\end{align*} 
Therefore part (a) of Proposition \ref{p:lwp:AV} is proved on $(-T,T)$, for sufficiently small 
$T<\min(\de_0,\de_1)$,  
 $\de_1=\de_1(p,d,\Vert \psi_0\Vert_{\sH^1})$. 
 
Moreover, we note that the solution $\psi$ belongs to $C(I, \sH^1)$ and  
equation \eqref{eqNls_AV} holds in the dual space $\sH\inv=(\sH^1)^*$, 
which follow from  \eqref{u=duhamel}. 

The proof of part (b) in Proposition \ref{p:lwp:AV} is an easy modification of the proof of part (a),
where the interval of local existence is dependent on some  $T=T(\psi_0)>0$, i.e., the profile of the initial data. 

(c) We show case (ii) only. If $G(x,|\psi|^2)\ge 0$,  we deduce from (\ref{EAV:Q}) and Lemma \ref{l:diamag}   
\begin{align*} 
&\Vert \psi\Vert^2_{\sH^1}\approx \Vert \nabla_A \psi\Vert^2_2+\Vert |V|^{1/2} \psi\Vert^2_2+\norm{\psi}_2^2\\
\lesssim&  E(\psi_0)+ C \int |\psi_0|^2 \,,
\end{align*}
where 
$ E(t)= \frac12 \Vert  \nabla_A \psi\Vert_2^2+ \int V(x) |\psi |^2+\int G(x, |\psi|^2 )$.
Noting that the local existence and uniqueness are valid for $T=T(\Vert \psi_0\Vert_{\sH^1})>0$ only,
  we obtain the global result 
by a simple bootstrapping argument.   
Cases (i) and 
(iii) can be proven via a similar argument based on the proof of part (a). 
In fact, to show the statement in Case (i), we can bound the $\sH^1$-norm of $\psi(t)$ for all time 
by using GN inequality.
For (iii), we can show that $\Phi$ is a contraction mapping for any fixed $T=T_0>0$,
where we only need to  assume $\ga=C\norm{\psi_0}_{\sH^1}<\veps$ is sufficiently small. 

(d) The mass and energy conservation laws.
To prove (\ref{mass-AV}) it requires  the gauge condition: 
$\Im (N(x,\psi),\psi )=0 $, 
where $(f,g):=\int f\bar{g}dx$ denotes the inner product in $L^2$.  
Indeed, one calculates using (\ref{eqNls_AV}) 
\begin{align*}
& i \frac{d}{d t} \int |\psi|^2 dx =2i \Im \int \bar{\psi} N(x,\psi) dx
=2i \Im (  N(x,\psi),\psi ). 
\end{align*}
But the gauge condition is true, because 
$ \bar{\psi} N(x,\psi) =  G'(x,|\psi|^2)|\psi|^2$ is real-valued.  

To prove (\ref{Eu:AV})            
for (\ref{eqNls_AV}), we write $E_{A,V}(\psi)=:E_{A,V}(t):=E_{lin}(t)+E_{nonlin}(t)$.
We show the above for (\ref{Eu:AV}) in  the general case $N(x,\psi)=G'(x,|\psi|^2)\psi$;
$E_{lin}(t)= \int \bar{\psi} H_{A,V}\psi $,  
and $E_{nonlin}(t)=\int G(x,|\psi|^2) dx$. 
  
Indeed, we write $E_{A,V}(\psi)=:E_{A,V}(t):=E_{lin}(t)+E_{nonlin}(t)$ such that
\begin{align*} 
&E_{lin}(t)= \int \bar{\psi} H_{A,V}\psi dx = \int \left(\frac12 |\nabla_A \psi|^2 + V(x)|\psi|^2\right)dx   \\
&E_{nonlin}(t)=\int G(x,|\psi|^2) dx \,.
\end{align*} 
 Multiply  $H_{A,V}\psi$ on both sides of (\ref{eqNls_AV}) and then add the conjugate. 
We obtain the following: For the linear part   
\begin{align*} 
    \frac{d}{dt} E_{lin}(t) %
    =& -i \int \left(  \overline{H_{A,V} ( \psi ) } G'(x,|\psi |^2) {\psi } - H_{A,V} (\psi)  G'(x,|\psi |^2) \bar{\psi}\right) dx\\
   =& -2 \Im (H_{A,V} \psi, G'(x,|\psi |^2) {\psi} ) \\
  =& { i \Im ( \De_A\psi,  G'(x,|\psi|^2) \psi)\,}.
\end{align*} 
For the nonlinear part,  
recall from (\ref{EAV:Q}) that $G(x,s)$ is real-valued and differentiable in $s$ 
\begin{align*}
    \frac{d}{dt}E_{nonlin}(t)&= \int G'(x,|\psi|^2) (\psi_t \overline{\psi}+ \psi \overline{\psi}_t) dx\\
    &= -i\int G'(x,|\psi|^2) \overline{\psi} \left( H_{A,V} (\psi)  +  G'(x,|\psi|^2)\psi\right)\\
    &+ i\int G'(x,|\psi|^2) \psi \left( \overline{H_{A,V} ( \psi) } +   G'(x,|\psi|^2)\overline{\psi}\right)\\
    &=  2 \Im (H_{A,V} \psi, G'(x,|\psi|^2)  {\psi} )=
    {-}\Im ( \De_{A} \psi, G'(x,|\psi|^2\psi).
\end{align*}
Therefore, (\ref{Eu:AV}) is true. \end{proof} 

(e) Blowup alternative statement is given in \cite{BHHZ23t,LeoZ22n};  esp. some  details of the proof can be found   in \cite{LeoZ22n} following the line of bootstrapping 
argument using Strichartz estimates in \cite{Cazenave2003}.  


\begin{proof}[Proof  of (e)]   Assume $T_+<\iy$ is finite \mbox{(respectively, $T_-<\iy$)} and 
$\norm{\psi}_{L^q( I_*, \sH^{1,r})}$ is finite. 
Let $0\leq t \leq t+\rho<T_+$. For fixed $t$ applying the operator $U(\rho)=e^{-i\rho\cL}$ to \eqref{Phi(u):duhamel} with $\Phi(\psi)=\psi$ followed by a change of variable gives 
\begin{align}  
& U(\rho)(\psi(t))=U(t+\rho)\psi_0- i\int_{-t}^0 U(\rho-s) N(\psi)(s+t,\cdot) ds. \label{Phi(u):duhamel1}
\end{align} 
Meanwhile, since $0 \leq t+\rho<T_+$, we have from the same formula \eqref{Phi(u):duhamel} 
\begin{align}  
& \psi(t+\rho)=U(t+\rho)\psi_0- i\int_{0}^{\rho} U(\rho-s) N(\psi)(s+t,\cdot) ds. \label{Phi(u):duhamel2}
\end{align} 
Next, we subtract \eqref{Phi(u):duhamel2} from \eqref{Phi(u):duhamel1}
\begin{align*}
   U(\rho)\psi(t,x)= \psi(t+\rho)+i\int_{0}^{\rho} U(\rho-s) N(\psi)(s+t,x) ds.
\end{align*}
Applying \eqref{inhomo-stri-L} to the second term we obtain 
\begin{align*} 
&\Vert U(\cdot) \psi(t,\cdot) \Vert_{L^q([0,T_{+}), \sH^{1})} \le \Vert \psi \Vert_{L^q([t,T_{+}), \sH^{1,r})} 
+ C_{d,q,r,\td{q},\td{r}} 
  \Vert F \Vert_{L^q([t,T_{+}), \sH^{1,r})}
\end{align*} 
for any $0<T_{+}-t<\eta$ for some small $\eta>0$. So, by virtue of \eqref{qr:stri-L}, there exists $t^{*}=t^{*}(\delta)>0$ such that
\[
\Vert U(\cdot) \psi(t,\cdot) \Vert_{L^q([0,T_{+}), \sH^{1})} < \delta 
\]
on the interval $[0,T_{+}-t+t^{*}]$. Now for fix $t$ and $t^{*}$ 
we apply the mapping $\Phi$ on the closed ball  %
\[
E_{[t,T_{+}+t^{*}]}:=\{u: \norm{u}_{L^q([t,T_{+}), \sH^{1,r} )}\leq 2\delta \}
\]
to obtain the existence of a solution that is defined on an interval beyond $[t, T_{+})$.
This contradiction shows 
that $\norm{ \psi }_{L^q([t,T_{+}), \sH^{1,r}) }$ is infinite. 

Next, we assume that $\limsup_{t\to T_+} \Vert \nabla \psi(t)\Vert_2< \infty.$ Then we can find a small $\tau_0>0$ such that the solution $\psi(t)$ extends beyond  $(T_+ -\tau_0,T_+)$ to  $(T_+ -\tau_0,T_+ +\tau_0)$, which contradicts with maximality of $T_+$. Therefore, $\lim_{t\to T_+} \Vert \nabla \psi(t)\Vert_2= \iy$. 

{To show  $\lim_{t\to T_+} \norm{\psi}_\iy=\iy$, we start with an estimation of the nonlinear term with $C>0$. 
\begin{align*}
    \left|\int G(x,|\psi|^2) \right| &\leq C \norm{\psi}^2_2  + C\int |\psi|^{p+1}\\
    &\leq C \norm{\psi_0}^2_2 + C \norm{\psi}_\iy^{p-1} \norm{ \psi_0}_2^2  \,.
\end{align*}  
The condition ($V_0$) implies that $V(x)$ is bounded from below, and so, 
 $\int V(x)|\psi|^2 \geq -C'\int |\psi |^2$ for some $C'>0$. 
 Thus we deduce from  energy conservation that 
\begin{align*}
    E_{A,V}(\psi_0)\geq \frac12 ||\nabla_A \psi ||_2^2 - C_1 \norm{\psi_0}^2_2 - C \norm{\psi }^{p-1}_{\iy} 
    \norm{\psi_0}_2^2     
\end{align*}
where $C_1>0$.  Since  $\lim_{t\rightarrow T_+} ||\nabla \psi||_2 = \iy$, 
it follows that $\lim_{t\rightarrow T_+} ||\psi||_\iy = \iy$.
} This completes the proof of Proposition \ref{p:lwp:AV}.
\end{proof} 

\medskip  
\begin{remark}  
{ Let $T_+$ be finite. Following \cite{BHZ19} we can show that for any 
 initial data there is a  lower bound for the blow-up rate. 
According to part (e) of Proposition \ref{p:lwp:AV}, 
the solution of \eqref{eqNls_AV} satisfies $\lim_{t \rightarrow T_+} \norm{\nabla \psi }_{2} = \infty$.
Then there exists  $C=C_{p,d}> 0$ such that
	\begin{align*}
	\left\Vert\nabla \psi(t) \right\Vert_{2}
	\geq C (T_+ - t)^{-(\frac{1}{p-1} - \frac{d-2}{4})}.
	\end{align*}
If  $p = 1 + \frac{4}{d}$,  then we have the lower bound 
	\begin{align*}
	\norm{ \nabla \psi(t) }_{2}\geq C (T_+ - t)^{-\frac{1}{2}}\,.
	\end{align*}
}
The same kind of lower bound holds for the situation $T_-<\iy$. 
\end{remark} 

\section{Existence of ground states for mNLS}\label{s:E-gss:AV}  
In this section, we  first prove Theorem \ref{T-Existence}   when $p=1+4/d$, which is concerned with
the existence of ground state solutions for (\ref{eq:AV}). 
Then, Theorem \ref{t2:orb-Oc} follows as a corollary, which asserts the stability for the set of magnetic g.s.s. $\cO_c$ in $\Sigma_{A,V}$.   
In doing so, we construct a ground state solution that is  a minimizer   for Problem (\ref{EAV:Ic})  
  under the condition $V(x)\to \iy$ as $|x|\to \iy$.   
In fact,  we will show the result  assuming the conditions in Hypothesis \ref{hy:G0-AV}. 
Recall form (\ref{EAV:Q}) that the associated energy functional $E_{A,V}$ is given by 
	\begin{equation}
	E_{A, V}(u)
	=\int \left(\frac12 |\nabla_A u|^2 +V(x) |u|^2 +G(x,|u|^2) \right) dx,\label{Eomga:u-G}
	\end{equation}
where  $G(x,s^2)=2\int_0^{s} N(x,\tau)d\tau$ for $s\ge 0$, 
and the nonlinearity $N$ satisfies  (\ref{G(xv)^p}).

\subsection{Proof of Theorem \ref{T-Existence}} \label{ss:proof-Thm1}
 We follow the argument described in  \cite[Section 7]{BHHZ23t}. 
  Note the following properties as required in the proof.   
 \begin{enumerate}
\item[(a)]   { The linear energy of $u$  in $S_c$
is equivalent to the norm of $u$ in $\Sigma$ space (modular the mass $\int |u|^2=c^2$)}.  
\item[(b)] Since $V(x)\to \iy$ as $|x|\to \iy$,  {$\Sigma$ is compactly embeded in $L^r$ for $r\in [2,2^*)$.} \end{enumerate} 

\begin{proof} Let $p=1+4/d$. {Condition ($V_0$) implies in particular 
\begin{align*}  
& V(x)\ge -C_1\qquad\text{for all $x$ and some constant $C_1>0$}. 
\end{align*} 
}
{From (\ref{G(xv)^p}) we deduce that  
\begin{align*}
& \int |G(x,|u|^2)| dx\le 2 \int \int_0^{|u|} \vert N(x,s)\vert ds dx\\
\le& C_0\int |u|^2+ \frac{2C_0}{p+1} \int |u|^{p+1}\,.
\end{align*}
}

Now, applying \eqref{eqGNA} we obtain from (\ref{Eomga:u-G}) that for all $u\in S_c$ and $c<\Vert Q\Vert_2/{C_0^{d/4}}$
\begin{align}
        E_{A,V}(u )  
        {\ge}& \frac12 \int |\nabla_A u|^2  + \int { (|V(x)|-2C_1)} |u|^2 {-C_0\int |u|^2} \notag\\ 
        -& {\frac{2C_0}{p+1} } C_{GN} \norm{\nabla_A u}_{2}^2\ \norm{u}_2^\frac{4}{d}\notag\\
 \ge&  \frac12\left(1-{C_0}\frac{\norm{u}_2^\frac{4}{d}}{\norm{Q}_2^{\frac{4}{d}}}\right) \norm{\nabla_A u}_2^2
 +{\int |V|\, |u|^2} {-(C_0+2C_1)\int |u|^2}\,. \label{Eu>Sigma} 
\end{align} 
The estimate (\ref{Eu>Sigma}) ensures that we can repeat the logical sequence in the proof of  \cite[Theorem 7.8]{BHHZ23t} as follows:
\begin{itemize}
    \item[(i)] For all $u\in S_c$, $E_{A,V}(u)$ is bounded from below by a constant. Hence, $I_c\ne -\iy$ in  (\ref{eqMinimizationProblem})     exists. 
    \item[(ii)] Let $\{u_n\}\subset S_c$ be any minimizing sequence for \eqref{eqMinimizationProblem}, then it is bounded in $\Sigma$.
    \item[(iii)] Since $\Sigma$ compactly embedded in $L^r$, $r\in\left[2,\frac{2d}{d-2}\right)$, there exist
     $\phi\in\Sigma\cap S_c$ and a subsequence of $(u_n)$, that we still denote as $\{u_n\}$, such that 
     $\lim_n u_n=\phi$ in $L^{p+1}$ and w-$\lim_n u_n=\phi$ in $\Sigma$.  
   \item[(iv)] As a consequence of (iii) we have,  up to some subsequence, 
  \begin{align*} 
    &\lim_n \int G (x, |u_n|^2)= \int G(x,|\phi|^2)\, \\ 
    and\quad & \norm{ \phi}_\Sigma\le \liminf_n \norm{ u_n}_\Sigma \,.   
 \end{align*}
    \item[(v)] It follows that 
    \[
    E_{A,V}(\phi) \le \lim_{n\rightarrow\infty} E_{A,V}(u_n) =I_c. 
    \]
By the definition of the minimum, we must have $E_{A,V}(\phi)=I_c$. 
 Moreover,  $u_n\rightarrow\phi$ in $\Sigma$-norm as $n\to \iy$. 
\end{itemize}
Therefore, there exists a minimizer of Problem \eqref{EAV:Ic}. 
\end{proof}

\begin{remark}\label{re:nonE>c}
Theorem \ref{T-Existence} suggests that in the focusing, mass-critical case, there exists ground state solutions on the 
level set  $S_c$ whenever $c$ is below  $\norm{Q_0}_2$. 
{However such solution(s) may not be { unique}. 
Esteban and Lions \cite{EsLion89} showed that there are infinitely many solutions of (\ref{eq:AV}) 
if $-\lam<|\Om|$ when $A=\Om(-x_2,x_1,0) $, $V=0$ 
and $G(x,|u|^2)= -\frac{2}{p+1}|u|^{p+1}$ (here ${ B}=\td{M}_3(-2\Om)=-2\Om\begin{pmatrix} 
  \sigma& {} \\  
{}   &0
\end{pmatrix} $, $\sigma=\begin{pmatrix} 
  0& {-1} \\  
{1}   &0
\end{pmatrix} $
is identified with ${\bf B}=\la 0,0,2\Om\ra$); 
also see \cite{ArioSzu03mag,DingWang2011mag} in the context of a general magnetic field}.    
If $c\ge \norm{Q_0}_2$, then  there  exist no  solutions 
for $(\ref{EAV:Ic})$, see \cite{GuoSeir2014,LeNamRou18a*} 
and \cite[Theorem A]{GuoLuoPeng23non}. 
\end{remark} 
\begin{remark} \label{rem:mini-compact}
From the outlined proof above we see that 
given any minimizing sequence $(u_n)$ to (\ref{eqMinimizationProblem})-(\ref{EAV:Ic}),  
there exists a subsequence $(u_{n_j})$ converging to some minimizer $\phi$ in $\Sigma=\Sigma_{A,V}$. 
This suggests that {\em any minimizing sequence to Problem (\ref{EAV:Ic}) is relatively compact in $\Sigma$}. 
\end{remark} 
\subsection{Euler-Lagrange equation for $Q_{A,V}$}\label{ss:EL-QAV} 
A solution of Problem (\ref{EAV:Ic}) is called a {\em magnetic ground state solution} denoted as $u=Q_{A,V}$.   
It is  a stationary solution $u$ of (\ref{eqNls_AV}) in $\Sigma=\Sigma_{A,V}$ satisfying  the Euler-Lagrange equation (\ref{eq:AV})  for some multiplier $-\lam$
\begin{equation*} 
-\lam u = -\frac12\De_A u + V(x) u + G'(x,|u|^2) u\, .  
\end{equation*} 

 Indeed, if $E_{A,V}(u)=I_c$, that is, $u$ is a minimizer restricted to the submanifold $S_c\subset \Sigma$,
  then $u$ is actually a global minimizer in the sense that 
 the action functional $S_{\nu}(u)$ attains the least energy value over  all non-trivial $\Sigma$-solutions of \eqref{eq:AV}, 
 where
\begin{align}\label{eSnu}  
S_{\nu}(u):&=\frac12 \int | (\nabla-iA)  u|^2+\int (V+\nu) |u|^2  +\int G(x,|u|^2) \nonumber\\
 & = E_{A,V}(u) + \nu \int|u|^2\,. 
\end{align}
\begin{prop}\label{p:Snu:gss} The function  $u\in \Sigma$ is a ground state solution of \eqref{eq:AV} with $\lam\in\R$
 if and only if for some $\nu\in\R$ 
\begin{equation*}
S'_{\nu}(u)=0\ {and}\ S_{\nu}(u)=\inf_{0\ne \phi\in \Sigma}\{S_{\nu}(\phi):  S'_{\nu}(\phi)=0 \}\,,
\end{equation*}   
in which case  
     $u=u_\nu$ solves \eqref{eq:AV} with $\nu=\lam$. 
 \end{prop} 
The description of $u=Q_{A,V}$  in Proposition \ref{p:Snu:gss} can be proven in a manner similar to that in which 
the fourth-order NLS is treated \cite[Theorem 6.1]{LZZ23t}. 

Alternatively,  let $\phi\in S_c$ be a minimizer of (\ref{EAV:Ic}).   
For any $w$ in the tangent space $T_\phi$ to the sphere $S_c$ in $\Sigma$ 
the vector $\phi+tw$ has square norm 
$\norm{\phi+tw}_2^2=c^2+t^2\norm{w}^2_2$.  
As $t\searrow 0$, it is a ``sequence'' of decreasing minimizing functions. 
Hence $\phi$ is a critical point of the energy functional  
$t\mapsto E_{A,V}(\phi+tw)$.
An easy calculation shows that, since $w\in \mathrm{span}\{u\}^\perp\cap \Sigma$, i.e., $( \phi,w)_{L^2}=0$,
\begin{align*} 
& \frac{d}{dt} E_{A,V}(\phi+t w)\big\vert_{t=0}\\
=&2\Re\left( -\frac12\De_A \phi+  V\phi + G'(|\phi |^2) \phi , w\right)_{L^2}=0\,.
\end{align*} 
This implies that for some constant $\nu\in\R$ 
\begin{align*}
 -\frac12\De_A \phi+  V\phi + G'(x,|\phi|^2) \phi =-\nu \phi\,. 
\end{align*}

\subsection{Proof of Theorem \ref{t2:orb-Oc} }\label{ss:orb-stab-u} 
 We  prove the orbital stability of the set of standing waves of \eqref{eqNls_AV}. 
 We will assume $A,V$ and $N$ satisfy Hypothesis \ref{hy:V0-B0-G1}, 
which  gives rise to the local existence theory in Proposition \ref{p:lwp:AV} 
 along with the mass-energy conservation laws (\ref{mass-AV})-(\ref{Eu:AV}). 
The line of proof 
follows from a standard concentration compactness argument.   
Comparing  the notions for orbital stability in this paper and \cite{CazEs88}, 
Definition \ref{def:orb-AV}  
is a mild stability result for the set of minimizers  while \cite{CazEs88} refers to  the set of minimizers with more specific symmetry about certain $m$-vortex type $z$-axial symmetric functions in $\Sigma_{A,0}=\sH^1_A$. 

\begin{proof}[Proof of Theorem \ref{t2:orb-Oc}]  
In order to show \eqref{eOS:magV}, we proceed  by contradiction argument. 
Suppose $\cO_c$ is not stable, then  there exist $w \in \cO_c$,
$\varepsilon_0 > 0$ and a sequence $\{ \psi_{n,0} \} \subset \Sigma$ such that
	\begin{equation}
	\| \psi_{n,0} - w \|_\Sigma \rightarrow 0 \ \text{as} \ n \rightarrow \infty \
	\text{but} \ \inf_{\phi \in \cO_c} \| \psi_n(t_n, \cdot) - \phi \|_\Sigma \geq \varepsilon_0
	\label{en:phi_0-Oc}
	\end{equation}
for some sequence $\{ t_n \} \subset \mathbb{R}$, 
where $\psi_n$ is the solution of \eqref{eqNls_AV} corresponding to the initial data $\psi_{n,0}$.  

Let $w_n := \psi_n(t_n, \cdot)$.   
Then $(w_n)$ is a minimizing sequence such that  
 $\| w_n \|_2 \rightarrow c$ and $E_{A, V}(w_n) \rightarrow I_c$ by  
the continuity of $E_{A,V}$ on $\Sigma$ and   the {conservation laws (\ref{mass-AV})-(\ref{Eu:AV})}.   
According to Remark \ref{rem:mini-compact},
  $\{ w_n \}$ contains a subsequence $(w_{n_j})$ converging  to some $\td{w}$ in $\Sigma$ with
$\| \td{w} \|_2 = c$ and $E_{A,V} (\td{w}) = I_c$. This means that  $\td{w} \in \cO_c$ 
and as a consequence  
\[
\inf_{\phi\in \cO_c} \| \psi_{n_j}(t_{n_j}, \cdot) - \phi \|_\Sigma \leq \| w_{n_j} - \td{w} \|_\Sigma\to 0,
\]
which contradicts \eqref{en:phi_0-Oc}.  
\end{proof}


\begin{remark}  
Theorem \ref{t2:orb-Oc} says  that if the initial data $\psi_0$  is sufficiently close to some 
$\vphi=Q_{A,V}\in \cO_c$  
for $c$ below $\norm{Q_0}_2$,  
then $\psi(t,\cdot)$ is close to some $\phi=\phi^t\in \cO_c$ for all time.     
It is known that the orbital stability of a single standing wave $e^{i\lam t}\vphi(x)$ is related to the uniqueness of
the ground state $\vphi$. 
Concerning the uniqueness for magnetic g.s.s. of (\ref{eq:AV}), an apparent symmetry is 
the phase invariant $e^{i\theta}\cO_c=\cO_c$ for $\theta\in\R$.
There is symmetry breaking:  Owing to the presence of $(A,V)$, 
translation, rotation and scaling invariance are lost. 
Relating Remark \ref{re:nonE>c}, the multiplicity of g.s.s. might leave the question open
regarding the  stability/instability for a single standing wave.  
Nevertheless, it is possible to study the stability property under certain constraint to   
some submanifold of $\Sigma$ when $V(x)$ admits cylindrical symmetry, say, cf. \cite{CazEs88,EsLion89}.  
\end{remark} 

\begin{remark} Let $p=1+4/d$. 
We can show a weaker version of the orbital stability of a single standing wave on any finite time interval.  
In fact,  suppose $\vphi_0\in \cO_c$,
$c<C_0^{-\frac{d}{4}}\norm{Q_0}_2$ is a ground state satisfying (\ref{eq:AV}).
We may assume $\psi_{0}$  is sufficiently close to $\vphi_0$ in $\Sigma$
so that a global solution exists. 
Let $\psi(t):=\psi(t,x)$ and $\vphi(t):=\vphi(t,x)=e^{i\lam t}\vphi_0$ be the corresponding solutions of 
Eq. (\ref{eqNls_AV}). 
For any fixed $T>0$ let $I_T:=[-T, T]$ and 
divide $I_T:=\cup_k I_k$, $I_k=[t_{k-1},t_k]$, $k=1,\dots, K_0$ 
such that all $|I_k|=\de$, $\de>0$ to be determined in a moment. 
In view of (\ref{u=duhamel}), if $t\in I_k$
\begin{align*}  
& \psi(t)-\vphi(t)=U(t)(\psi_{0}-\vphi_0)- i\int_{t_{k-1}}^t U(t-s) \left( N(x,\psi)-N(x,\vphi )\right) ds .
\end{align*}  
In light of (\ref{qr:stri-L}) and (\ref{eZ:u-v}), 
we have if $|I_k|<\de:=\de(\norm{\vphi_0}_\Sigma) >0$ sufficiently small, then 
\begin{align*}  
&\norm{ \psi(t)-\vphi(t)}_{Z_k}\le C\norm{\psi_{0}-\vphi_0}_\Sigma
+   \frac12\norm{ \psi(t)-\vphi(t)  }_{Z_k} 
\end{align*}  
where $Z_k:=L^\iy(I_k, L^2)\cap L^q(I_k, L^r) $.  
This implies that 
\begin{align}  
&\norm{ \psi(t)-\vphi(t)}_{L^\iy(I_T, L^2)\cap L^q(I_T, L^r) }\le 2CK_0 \norm{\psi_{0}-\vphi_0}_\Sigma\,.
\end{align}  
Thus we have shown that  given any interval $I_T=[-T,T]$, the solution $\psi(t,\cdot)$ stays  close to $e^{i\lam t}\vphi_0$ 
on $I_T$ with respect to the metric $L^\iy(I_T, L^2)\cap L^q(I_T, L^r) $ as soon as $\psi_0$ is sufficiently close to $\vphi_0$ in $\Sigma$.  
\end{remark}  


\section{Non-existence of g.s.s. in fast rotation regime}\label{s:non-E:AV}
In this section we show Theorem \ref{t3:non-E}, which is a non-existence result 
of g.s.s. for (\ref{eq:AV})  when the condition $(V_0)$  is not verified. 
In fact, we  consider the RNLS (\ref{eU:Om-ga}) where the condition $(N_0)$ is assumed as in Hypothesis \ref{hy:G0-AV}. 
We shall construct a sequence of functions $(\psi_m)\subset S_1$ 
such that $E_{\Om,\ga}(\psi_m)$ tends to negative infinity.
Then,  Problem (\ref{EAV:Ic}) does not admit a solution since $I_c=-\iy$ is not achievable. 
In the first scenario let $\ga_j>0$ for all $j$ and $|\Om|> \min(\ga_1,\ga_2)$, say. 
In view of (\ref{E3:OmV}) 
 and $(\ref{ecase:VgaOm})$,  
it is sufficient to take up the 3D case where $A=-\Om x^\perp$, 
$V_\ga(x)=\frac12\sum_j \ga_j^2x_j^2$ so that 
$V(x)=\frac12\sum_{j=1}^2 (\ga_j^2-\Om^2)x_j^2 + \frac12 \ga_3^2x_3^2$.
If $d=2$, we can simply let $\ga_3=0$. 
The other scenario is $\ga_{j_0}<0$ for some $j_0$, 
in which case $V_\ga$ is repulsive on the $x_{j_0}$-direction. 
Evidently, in both scenarios 
 $V(x)\not\rightarrow \iy$. 
In particular, our examples include the case  of 
 an anisotropic quadratic potential $V_{\td{\ga}}$, $\td{\ga}=(\ga_1,\ga_2,\ga_3)\in\R^3$, $\ga_1\ne \ga_2$, cf. \eqref{eV:ga-sgn}.    

\subsection{$V_\ga(x)=\frac{\ga^2}{2}|x|^2$ isotropic, $d=2,3$} \label{ss:non-E:OmV-isotropic}  
If $|\Om|> \ga:=\ga_1=\ga_2>0$,  we recall a 2D example  from  \cite{BHHZ23t},  see also \cite{BaoCai2013num}. 
 Define  for $m\in\Z$ and $x=(r,\theta)$ using polar coordinates 
\begin{align} 
&\psi_m(x)=\frac{\ga^{(\frac{|m|+1}{2})}}{\sqrt{\pi (|m|)!}} |x|^{|m|} e^{-\ga |x|^2/2} e^{im\theta}\,. \label{ega:psi_m} 
\end{align}  
Then $\norm{\psi_m}_2=1$ for all $m\in \Z$.  
We compute
\begin{align*}
& \pa_{x_1} (\psi_m(r,\theta)) =\frac{\ga^{\frac{|m|+1}{2}}}{\sqrt{\pi (|m|)!}}  e^{im\theta} r^{|m|-1} e^{-\ga r^2/2}\left( (|m|-\ga r^2)\cos\theta-i m \sin\theta \right)  \\
& \pa_{x_2} (\psi_m(r,\theta)) =\frac{\ga^{\frac{|m|+1}{2}}}{\sqrt{\pi (|m|)!}}  e^{im\theta} r^{|m|-1} e^{-\ga r^2/2}\left( (|m|-\ga r^2)\sin\theta+i m \cos\theta \right).
\end{align*}
Hence,
\begin{align*} & \int |\nabla \psi_m|^2 dx =2\pi\int_0^\iy \frac{\ga^{|m|+1}}{\pi (|m|)!}  r^{2|m|-2} e^{-\ga r^2}  ( (|m|-\ga r^2)^2+m^2 )rdr 
= (|m|+1)\ga.  \qquad  
\end{align*}

 It is also easy to calculate 
\begin{align*}
&\int V_\ga(x) |\psi_m|^2 dx={2\pi}\int_0^\iy(\frac{\ga^2}{2} r^2  |\psi_m|^2) rdr =\frac{|m|+1}{2}\ga \notag\\
& -i \Om \int ({ x}^\perp \cdot\nabla \psi_m) \overline{\psi_m}dx 
=i \Om\int_0^{2\pi} d\theta\int_0^\iy (\overline{\psi_m} \pa_\theta\psi_m ) rdr\notag\\
=&2\pi i \Om\frac{\ga^{|m|+1}}{\pi (|m|)!}\int_0^\iy r^{2|m|} e^{-\ga r^2} (im) rdr
= -m\Om . 
\end{align*} 
And, if $G(|u|^2)\le C_0 |u|^{2}+ \frac{2C_0}{p+1}|u|^{p+1}$, $p\in (1,\iy)$
\begin{align*}
&\int G(|\psi_m|^2)dx=C_0+\frac{2C_0}{a} \ga^{\frac{a}{2}-1} \pi^{\frac{3}{2}(1-\frac{a}{2})}  \frac{1}{2^{(\frac{a}{4}-1)}\sqrt{a} (\sqrt{|m|})^{\frac{a}{2}-1}} +o(1)  \,,
\end{align*}
where $a=p+1$, 
$i A\cdot \nabla 
=i\Om\pa_\theta$ and $-x_2\pa_{x_1}+x_1\pa_{x_2}=\pa_\theta$.  
Thus we obtain in view of (\ref{EAV:Q}), (\ref{G(xv)^p}) 
 \begin{align*}
&E_{\Om,\ga}(\psi_m) 
= \frac{|m|+1}{2}\ga+\frac{|m|+1}{2}\ga-m\Om  +C_0 + o(1)\\ 
=&(|m|+1)\ga- m \Om  +C_0+o(1). \quad 
\end{align*} 
This suggests that $E_{\Om,\ga}(\psi_m)\to-\iy$ as $|m|\to \iy$ and $m\Om>0$. 

{If $d=3$,  $\Om>\ga$, $\ga:=\ga_1=\ga_2=\ga_3>0$, using cylindrical coordinates $(r,\theta,x_3)$ we}  
{  define
\[
\phi_m(r,\theta, x_3)= {\left( \frac{\gamma}{\pi} \right)^{\frac14}} \psi_m(r,\theta) e^{-\frac{\gamma x_3^2}{2}}\,.  
\]
Then $\norm{\phi_m}_2=1$ for all $m\in \Z$. We compute
\begin{align*}
 \pa_{x_1} (\phi_m(r,\theta,x_3)) &={\left( \frac{\gamma}{\pi} \right)^{\frac14}}
 \pa_{x_1}(\psi_m(r,\theta)) e^{-\frac{\gamma x_3{^2}}{2}}\\
 \pa_{x_2} (\phi_m(r,\theta,x_3)) &={\left( \frac{\gamma}{\pi} \right)^{\frac14}}
 \pa_{x_2}(\psi_m(r,\theta)) e^{-\frac{\gamma x_3^2}{2}}\\
 \pa_{x_3} (\phi_m(r,\theta,x_3)) &=-{\left( \frac{\gamma}{\pi} \right)^{\frac14}} 
 \psi_m(r,\theta) \gamma x_3 e^{-\frac{\gamma x_3{^2}}{2}} 
\end{align*}
Using calculations for $\psi_m$ we obtain
\begin{align*}  \int |\nabla {\phi}_m|^2 dx &= {\left( \frac{\gamma}{\pi} \right)^{\frac12}} \int |\nabla \psi_m|^2 e^{-\gamma x_3{^2}} dx + \int \ga^2 x_3^2 |\phi_m|^2 dx = (|m|+1)\ga + \ga^{3/2} \\
\int V_\ga(x) |\phi_m|^2 dx&= {(\frac{\ga}{\pi})^{\frac12}} \int V_\ga(x_1,x_2) |\psi_m|^2 e^{-\gamma x_3{^2}} dx + {\left(\frac{\ga}{\pi}\right)^{\frac12}} \int \frac{\ga^2 x_3^2}{2} |\psi_m|^2 e^{-\gamma x_3{^2}} dx\\
&= \frac{(|m|+1)\ga+\ga^{3/2}}{2}\\
-i \Om \int ({ x}^\perp \cdot\nabla \phi_m) \overline{\phi_m}dx&=-i \Om \int ({ x}^\perp \cdot\nabla \psi_m) \overline{\psi_m}dx = -m\Omega\,,
\end{align*}
and, 
{
\begin{align*}
&\int G(|\phi_m|^2)dx=C_0+\frac{C_2(a,\gamma)}{(\sqrt{|m|})^{\frac{a}{2}-1}} +o(1)  
\end{align*}
}
where $a=p+1$. Thus, 
 \begin{align*}
E_{\Om,\ga}(\phi_m)&= (|m|+1)\ga- m \Om + {O}(1).
\end{align*} 
}
Therefore, $E_{\Om,\ga}(\psi_m)\to-\iy$ as $|m|\to \iy$ and $m\Om>0$.



 
\subsection{$V_\ga$ non-isotropic:  $|\Om|>\ga_1=\min(\ga_1,\ga_2)>0$, $d=3$} \label{ss:non-ex:Om>ga1}  
Let the nonlinearity $N(u)=N(x,u)$ verify  (\ref{G(xv)^p}) with $p\in (1,\frac{d+2}{d-2})$.    
We define a sequence of functions $(\psi_m)_{m\in\Z}\subset\Sigma$ of vortex type
using the cylindrical coordinates $(x_1,x_2,x_3) = (r, \theta, x_3)$
\begin{equation*}\label{E:psi:3d}
    \psi_{m}(x)= C_m x_1^{|m|}e^{-\frac{\gamma(x)}{2}} e^{im\theta}\,,  
\end{equation*}
where $  \gamma(x)= \gamma_1 x_1^2 + \gamma_2 x_2^2+ \gamma_3 x_3^2$ and $C_m$ is chosen as 
\begin{equation*}
    C_m= \sqrt{\frac{\sqrt{\gamma_2\gamma_3}}{\pi}\frac{\gamma_1^{|m|+\frac{1}{2}}}{\Gamma(|m|+\frac{1}{2})}}
\end{equation*} 
so that \begin{align} \label{Cm:normal}
& |C_m|^2 \frac{\Gamma( |m|+\frac{1}{2}) }{\gamma_1^{ |m|+\frac{1}{2} } } \frac{\Gamma(\frac12) }{ \gamma_2^{\frac{1}{2}} } \frac{\Gamma(\frac12) }{ \gamma_3^{\frac{1}{2}} }=1.
\end{align}

Then for all $m\in\Z$
\begin{equation*}
    \begin{split}
        \norm{\psi_m}_2^2&= C_m^2\int_{\R^3} x_1^{2|m|}e^{-\gamma(x)} dx\\ 
        &= C_m^2\int_{\R} x_1^{2|m|}e^{-\gamma_1x_1^2} dx_1 \int_{\R} e^{-\gamma_2x_2^2} dx_2 \int_{\R} e^{-\gamma_3x_3^2} dx_3\\
        &=\frac{\sqrt{\gamma_2\gamma_3}}{\pi}\frac{\gamma_1^{|m|+\frac{1}{2}}}{\Gamma(|m|+\frac{1}{2})} \frac{\Gamma(|m|+\frac{1}{2})}{\gamma_1^{|m|+\frac{1}{2}}} \sqrt{\frac{\pi}{\gamma_2}} \sqrt{\frac{\pi}{\gamma_3}}=1.
    \end{split}
\end{equation*}
We see that $C_m$  are normalizing constants. 
Next, we consider the energy (\ref{E3:OmV})
\begin{equation*}
   \begin{split}
        E_{\Omega,\ga}(\psi_{m})&=\frac{1}{2}\int \vert\nabla \psi_m\vert^2 + \int V_\ga(x)|\psi_{m}|^2 \\ 
        +&\int G(|\psi_m|^2) +\int \overline{\psi_{m}} L_A (\psi_{m}) .  
   \end{split}
\end{equation*}
It is easy to calculate that 
\begin{equation}\label{eqGrad}
    \nabla \psi_{m} = C_m x_1^{|m|-1}e^{-\frac{\gamma(x)}{2}} e^{im\theta} \begin{bmatrix}
|m|-\gamma_1 x_1^2-\frac{i m x_2 x_1}{x_1^2+x_2^2}\\
-\gamma_2 x_1x_2+\frac{i m x_1^2}{x_1^2+x_2^2}\\
-\gamma_3 x_3 x_1
\end{bmatrix}
\end{equation}
and 
 $( L_A\psi_m,\psi_m)= - i\Om (x^\perp\nabla\psi_m,\psi_m)$ such that   
\begin{equation}\label{eqLA}
    \begin{split}
        - (x^\perp\cdot\nabla \psi_m)\overline{\psi_m} &=\overline{\psi_m} 
        \begin{bmatrix} -x_2\\   x_1\\ 0 \end{bmatrix}
 \cdot\nabla \psi_{m}\\
 &=  |C_m|^2 x_1^{2|m|-1}e^{-\gamma(x)} [(\gamma_1-\gamma_2)x_1^2x_2- |m| x_2 + im x_1]\,.
    \end{split}
\end{equation}
So, for the kinetic  energy term we deduce from \eqref{eqGrad} 
\begin{equation*}
    \begin{split}
         \norm{\nabla \psi_m}_2^2&= C_m^2\int x_1^{2|m|-2}e^{-\gamma(x)} 
         \Big[(|m|-\gamma_1x_1^2)^2 + \frac{m^2 x_1^2}{x_1^2+x_2^2}  \\
         & +(\gamma_2x_2x_1)^2+(\gamma_3x_3x_1)^2\Big] dx\\
        &\leq C_m^2 \int\left[ m^2  x_1^{2|m|-2}e^{-\gamma(x)}  -2\gamma_1|m| x_1^{2|m|}e^{-\gamma(x)} + \gamma_1^2 x_1^{2|m|+2}e^{-\gamma(x)} \right.\\
        &\left.+m^2x_1^{2|m|-2}e^{-\gamma(x)} + \gamma_2^2 x_1^{2|m|}x_2^2e^{-\gamma(x)}+ \gamma_3^2 x_1^{2|m|}x_3^2e^{-\gamma(x)}\right]dx\\
        &= \frac{\gamma_1 m^2}{|m|-\frac12} - 2\gamma_1 |m| +\gamma_1\left(|m|+\frac12\right) +\frac{\gamma_1m^2}{|m|-\frac12} + \frac{\gamma_2+\gamma_3}{2}\\
        &= \frac{\gamma_1 m^2 + |m|\gamma_1-\frac{\gamma_1}{4}}{|m|-\frac12}  + \frac{\gamma_2+\gamma_3}{2} \,,
    \end{split}
\end{equation*}
where we have used (\ref{Cm:normal}) and the identities $\Ga(\frac12)=\sqrt{\pi}$ and
\begin{equation}\label{ey:n-alpha}
\int_{-\iy}^\iy y^{2n}e^{-\al y^2}dy=\frac1{\al^{n+\frac12} } \Ga(n+\frac12)\,,\ \qquad \al>0,\ n>-\frac12\,. 
\end{equation}
It follows that
\begin{align*}
& \frac12\int |\nabla\psi_m|^2dx\le \frac{\ga_1}{2} |m|+O(1)\quad\ as\ |m|\to \iy.
\end{align*}
In a similar manner, for the potential energy term we obtain
\begin{equation*}
    \begin{split}
        \int V_\ga(x) |\psi_m|^2dx&=\frac{\gamma_1|m|}{2} + \frac{\gamma_1+\gamma_2+\gamma_3}{4}\,.
    \end{split}
\end{equation*}
For the nonlinear energy term, we obtain, as $m\rightarrow\infty$
\begin{equation*} 
 \int G(|\psi_m|^2) =O(1).
\end{equation*}
For the angular momentum energy, using  \eqref{eqLA} and  $\int_{-\iy}^\iy x_2 e^{-\gamma_2x_2^2}dx_2 = 0$ 
 we obtain
\begin{align*}
        &-i\Omega \int \overline{\psi_m} x^\perp\nabla (\psi_{m})  \\
   &=i \Omega|C_m|^2 \int x_1^{2|m|-1}e^{-\gamma(x)}  ( i m x_1)= -m \Omega.
       \end{align*}
Combining all terms together we obtain
\begin{equation*}
    E_{\Omega,\ga}(\psi_{m})\leq \gamma_1|m| - \Omega m  + O(1).
\end{equation*}
Therefore, if $|\Omega|>\gamma_1>0,$ then  $E_{\Omega,\ga}(\psi_m) \rightarrow -\infty$ as 
$m\rightarrow \sgn(\Om)\infty$.  Thus, we conclude that Problem \ref{EAV:Ic} does not admit a ground state solution.   

\subsection{$V_\ga$ partially repulsive:  $\ga_{j_0}<0$, $\Om\in \R$, $d=3$}\label{ss:Vga:repulsive} 

We may assume $\ga_3<0$ and $\ga_1,\ga_2>0$.  
The case $\ga_1<0$ or $\ga_2<0$ can be dealt with similarly. 
By analogy with the previous subsection we define a sequence of  functions $(\psi_m) \subset\Sigma$ 
\begin{equation*}
    \psi_{m}(x)= C_m x_3^{|m|}e^{-\frac{\gamma(x)}{2}}\,, 
\end{equation*}
where $ \gamma(x)= \gamma_1 x_1^2 + \gamma_2 x_2^2+ |\gamma_3| x_3^2$ and 
$C_m$ are normalizing constants such that
\begin{align}\label{eCm:ga3}
& |C_m|^2 \frac{\Gamma( |m|+\frac{1}{2}) }{|\gamma_3|^{ |m|+\frac{1}{2} } } \frac{\Gamma(\frac12) }{ \gamma_1^{\frac{1}{2}} } \frac{\Gamma(\frac12) }{ \gamma_2^{\frac{1}{2}} }=1.
\end{align}
Then we have $\norm{\psi_m}_2=1$. 
Also, we calculate that 
\begin{equation*}
    \nabla \psi_{m} = C_m x_3^{|m|-1}e^{-\frac{\gamma(x)}{2}} \begin{bmatrix}
-\gamma_1 x_1 x_3\\
-\gamma_2 x_2x_3\\
|m|-|\gamma_3| x_3^2
\end{bmatrix}
\end{equation*}
and 
 $( L_A\psi_m,\psi_m)=-i\Om ( x^\perp\nabla\psi_m, \psi_m)$ so that 
\begin{align*} 
-( x^\perp\nabla\psi_m, \psi_m) 
&=\overline{\psi_m} \begin{bmatrix} 
-x_2\\
x_1\\
0
\end{bmatrix} \cdot \nabla \psi_{m}\\
 &=  C_m^2 x_3^{2|m|}e^{-\gamma(x)} (\gamma_1-\gamma_2)x_1x_2\,.
\end{align*} 
In order to estimate  the energy $E_{\Omega,\ga}(\psi_m)$ in (\ref{E3:OmV}),   
  we compute by direct integration using (\ref{ey:n-alpha}) and (\ref{eCm:ga3}) 
\begin{equation*} 
    \begin{split}
     \norm{\nabla \psi_m}_2^2&= C_m^2\int x_3^{2|m|-2}e^{-\gamma(x)} 
     \Big[(\gamma_1x_3x_1)^2+  (\gamma_2x_2x_3)^2 + (|m|-|\gamma_3|x_3^2)^2 \Big] \\
        &=\left(\frac{|\gamma_3| m^2}{|m|-\frac12}-2|\gamma_3| |m|
        +|\gamma_3| (|m|+\frac12)+\frac{\gamma_1+\gamma_2}{2}\right)\\
 &=\frac{\gamma_1+\gamma_2+|\ga_3|}{2}+\frac{|\ga_3| |m|}{2|m|-1}\,. 
    \end{split}
\end{equation*}
For the potential energy we have
{
\begin{equation*}
    \begin{split}
        \int V_\ga(x) |\psi_m|^2&= \frac{C_m^2}{2}\int\left[ \gamma_1^2 x_1^2x_3^{2|m|} + \gamma_2^2 x_2^2x_3^{2|m|} - \gamma_3^2 x_3^{2|m|+2}\right]e^{-\gamma(x)} \\
        &=-\frac{ |\gamma_3| |m|}{2} + \frac{\gamma_1+\gamma_2-|\gamma_3|}{4}\,.
    \end{split}
\end{equation*}
}
For the nonlinear energy, 
\begin{equation*} 
 \int G(|\psi_m|^2) =O(1)\qquad as\ |m|\to \iy.
  \end{equation*}  
  Finally, 
  the angular momentum energy
\begin{equation*}
    \begin{split}
        -i\Omega \int \overline{\psi_{m}} x^\perp\nabla (\psi_{m}) 
        &=i\Om  C^2_m \int x_3^{2|m|}e^{-\gamma(x)} (\gamma_1-\gamma_2)x_1x_2 =0.
    \end{split}
\end{equation*} 
Combining all terms together  we obtain
\begin{equation*}
    E_{\Omega,\ga}(\psi_{m}) = -\frac{|\gamma_3| |m|}{2}  + O(1)\to -\iy \qquad as\  |m|\rightarrow \infty.
\end{equation*}  

{
The case $\ga_1<0$ or $\ga_2<0$ requires some modification of $\psi_m$. So, for $\ga_1<0$ we define
\begin{equation*}
    \psi_{m}(x)= C_m x_1^{|m|}e^{-\frac{\gamma(x)}{2}}\,, 
\end{equation*}
where $ \gamma(x)= |\gamma_1| x_1^2 + \gamma_2 x_2^2+ \gamma_3 x_3^2$ and 
$C_m$ are normalizing constants such that
\begin{align*}
& |C_m|^2 \frac{\Gamma( |m|+\frac{1}{2}) }{|\gamma_1|^{ |m|+\frac{1}{2} } } \frac{\Gamma(\frac12) }{ \gamma_3^{\frac{1}{2}} } \frac{\Gamma(\frac12) }{ \gamma_2^{\frac{1}{2}} }=1.
\end{align*}
Similarly, for $\ga_2<0$ we define
\begin{equation*}
    \psi_{m}(x)= C_m x_2^{|m|}e^{-\frac{\gamma(x)}{2}}\,, 
\end{equation*}
where $ \gamma(x)= \gamma_1 x_1^2 + |\gamma_2| x_2^2+ \gamma_3 x_3^2$ and 
$C_m$ are normalizing constants such that
\begin{align*}
& |C_m|^2 \frac{\Gamma( |m|+\frac{1}{2}) }{|\gamma_2|^{ |m|+\frac{1}{2} } } \frac{\Gamma(\frac12) }{ \gamma_3^{\frac{1}{2}} } \frac{\Gamma(\frac12) }{ \gamma_1^{\frac{1}{2}} }=1.
\end{align*}
Note that for both cases the angular momentum energy still vanishes and all other calculations just change $O(1)$. Therefore, we get 
\begin{equation*}
    E_{\Omega,\ga}(\psi_{m}) = -\frac{|\gamma_1| |m|}{2}  + O(1)\to -\iy \qquad as\  |m|\rightarrow \infty
\end{equation*}
for $\gamma_1< 0$ and
\begin{equation*}
    E_{\Omega,\ga}(\psi_{m}) = -\frac{|\gamma_2| |m|}{2}  + O(1)\to -\iy \qquad as\  |m|\rightarrow \infty
\end{equation*}
for $\gamma_2< 0$. 
}

{
Also, we consider the case when several gammas are less than $0$. So, we choose $\psi_m$ that corresponds to one of $\gamma < 0$, but with slight modification $ \gamma(x)= |\gamma_1| x_1^2 + |\gamma_2| x_2^2+ |\gamma_3| x_3^2$. This change keeps $||\psi_m||_2=1$ and allows us to repeat calculations in previous cases to get $E_{\Omega,\ga}(\psi_{m})\to -\iy $ as $|m|\rightarrow \infty$. 
}

{Finally, if $\gamma_{j_1}<0$ and $\gamma_{j_2}=0$, then we define 
\begin{equation*}
    \psi_{m}(x)= C_m x_{j_1}^{|m|}e^{-\frac{\gamma(x)}{2}}\,, 
\end{equation*}
where $ \gamma(x)= |\gamma_{j_1}| x_{j_1}^2 + x_{j_2}^2+ |\gamma_{j_3}| x_{j_3}^2$ and 
$C_m$ are normalizing constants such that
\begin{align*}
& |C_m|^2 \frac{\Gamma( |m|+\frac{1}{2}) }{|\gamma_{j_1}|^{ |m|+\frac{1}{2} } }  \Gamma(\frac12) 
 \frac{\Gamma(\frac12) }{ |\gamma_{j_3}|^{\frac{1}{2}} } =1.
\end{align*}
Such changes have effects only one constant $O(1)$ in the energy representation. Therefore, after repeating previous calculations we have
\begin{equation*}
    E_{\Omega,\ga}(\psi_{m}) = -\frac{|\gamma_{j_1}| |m|}{2}  + O(1)\to -\iy \qquad as\  |m|\rightarrow \infty.
\end{equation*}
}


\section{Numerical simulations for ground states near threshold} \label{s:numQAV}
We present numerical results on the  profile solution of the time-independent RNLS (\ref{eU:Om-ga}) 
 with $p=3$, $d=2$ and   $x^\perp=\la x_2,-x_1\ra$  
\begin{equation}\label{eQ:OmGa} 
 -\frac{1}{2} \nabla^2 u+ V_\ga(x)u +\mu |u|^{p-1}u - i \Om x^\perp\cdot\nabla u=-\lam u 
	\end{equation}  
where we specify $N(u)=\mu |u|^{p-1} u$, $\mu=-1$. 
The numerics is carried out with varying parameters $\Om$, $\gamma=(\ga_1,\ga_2)$ and    
initial data below and above the  g.s.s.  $Q_0$ and $Q_{\Om,\ga}$ separately.   
To verify the analytical results in Theorem \ref{T-Existence} and \cite{LeoZ22n} about the existence 
and  dynamical (strong) instability near the ground state, the following cases  are examined  of \eqref{eQ:OmGa} and (\ref{RNLS_Omga}),  
respectively,  when
\begin{enumerate}
\item[(i)] $\Om=0.5$, $(\ga_1,\ga_2)=(1,\sqrt{2})$, or $(2,8)$  
\item[(ii)] $\Om=0.8$, $(\ga_1,\ga_2)=(1,\sqrt{2})$, or $(1,2)$.  
\end{enumerate} 
Recall that when $V_\ga$ is isotropic and $N(u)=-|u|^{p-1} u$, we have the following threshold result
in view of \cite[Theorems 1.1]{BHHZ23t}. 
 
\begin{proposition}[]\label{t:gwp-blup-Om} 
Suppose $\psi_0\in \sH^1$. 
\begin{enumerate} 
\item[(a)] If $\Vert \psi_0\Vert_2< \Vert Q_0\Vert_2$, then (\ref{RNLS_Omga}) has a global  solution $\psi_0\mapsto \psi(t)$ in $\sH^1$. 
\item[(b)] Given any $c>1$,  there exists $\psi_0$ in $\sH^1$ with $\Vert \psi_0\Vert_2=c \Vert Q_0\Vert_2$ 
 such that $\psi_0\mapsto \psi(t)$ blows up  on some finite interval $[0,T_+)$, 
 i.e., $\Vert \nabla \psi(t)\Vert_2\to \iy $ as  $t\to T_{+}$. 
\end{enumerate} 
\end{proposition} 

\bigskip 
\subsection{Scaled profile $Q_0$ and simulations $p=3, d=2$} \label{ss:Q0-simulat} 
Let $Q_0$ be the free ground state solution of (\ref{ground}). 
Then, according to Theorem \ref{T-Existence}, $c_0:=\norm{Q_0}_2$ 
is the sharp threshold for \eqref{eQ:OmGa}, where  
the mass $M(Q_0)=c_0^2=5.85043$, cf. \cite[Eq.(1.12)]{LeoZ22n}. 
Figures \ref{Q0:Om0ga0}-\ref{f:belowQ:Om0.5-ga2(1,2)} 
show that the solution $\psi(t)$ exists globally in time when $\psi_0$ is slightly below $Q_0$ with respect to the mass. 
Figure \ref{f:aboveQ0:Om0.5-ga2(1,2)} shows 
blowup solution occurs in a short time 
when $\psi_0$ is only slightly above $Q_0$.  
The same paradigm occurs as is exhibited in Figures \ref{fig:0.99Q}-\ref{f:0.99Q-global}, 
and Figure \ref{1.03Q:Om0.8-ga2(1,4)}. 
Thus  the threshold statements in Proposition \ref{t:gwp-blup-Om} are  numerically verified in our simulations.  


To simulate  the dynamics near  $Q_0$, we design the numerical solver program 
based on a compilation of GPELab and 
ground state solution coding in \cite{AntDub2014gs} and \cite{Fibich_2015}. 
In the implementation, we first solve the ground state $Q=Q_0$ using the shooting method. 
Then, we use the scaled version of the saved g.s.s $Q_0$ 
 to run the evolution of the flow $\psi_0\mapsto \psi(t)$.  
The results on the initial time $t=0$ and the ending time  $t=2$  
are recorded and plotted in Figures \ref{psi(t):belowQ0}-\ref{f:aboveQ0:Om0.5-ga2(1,2)}.  
Examples are run per two set of parameters $[\Om,\ga,c, Q]$ respectively given by:
\begin{enumerate}
\item[(i)]  $ [0.5, (1,\sqrt{2}), 0.98, Q_0]$, 
and  $[0.5, (1,\sqrt{2}), 1.02, Q_0]$. 
\item[(ii)] $ [0.8, (1,{2}), 0.99, Q_0]$ and $ [0.8, (1,{2}), 1.03, Q_0]$. 
\end{enumerate}
The initial state is chosen as scaled $Q_0$ so that $M(\psi_0 ) = c^2\, M(Q_0)$, 
see \mbox{Figure \ref{gss_belowQ0}} to Figure \ref{1.03Q:Om0.8-ga2(1,4)}.   
Here the profile $Q=Q_0$ 
 can be numerically solved  using the shooting method from  \cite[pp.125-145]{Fibich_2015}.  

 \begin{figure}[H]  
\includegraphics[width=0.48\textwidth]{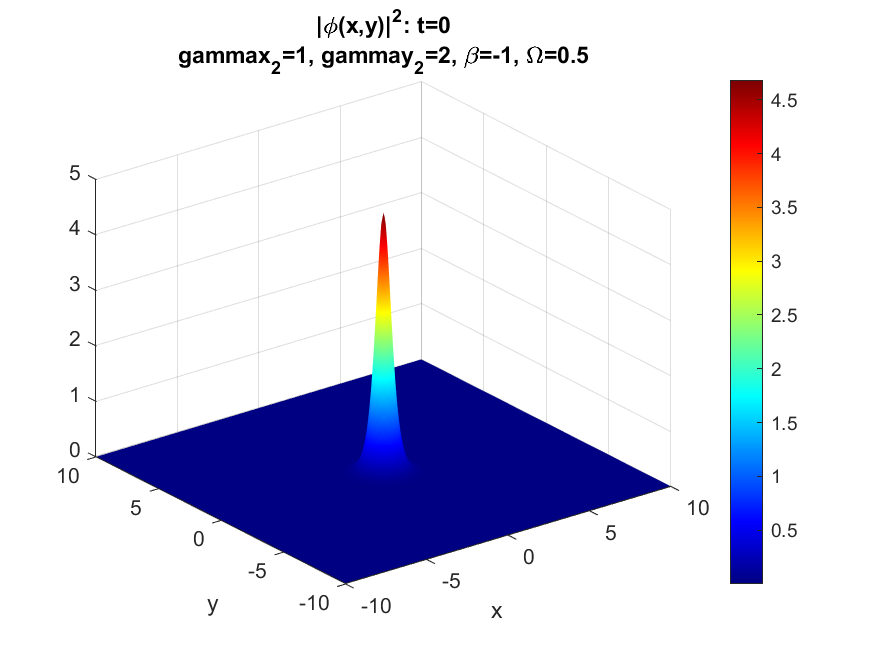}
\includegraphics[width=0.480\textwidth]{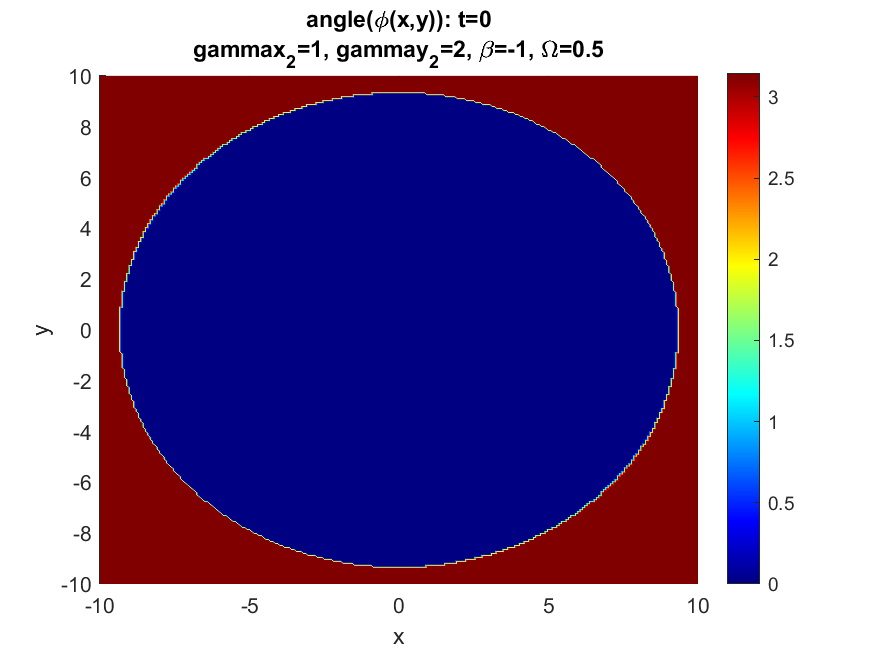} 
\caption{\label{gss_belowQ0} Scaled g.s.s.  $\psi_0=0.98Q_0$.  
 Left: plot of the modular $|\psi_{0}|^2$.\quad  
Right: plot of the phase for $\psi_0$} 
\label{Q0:Om0ga0} 
\end{figure}   
 
\begin{figure}[H]  
\includegraphics[width=0.48\textwidth]{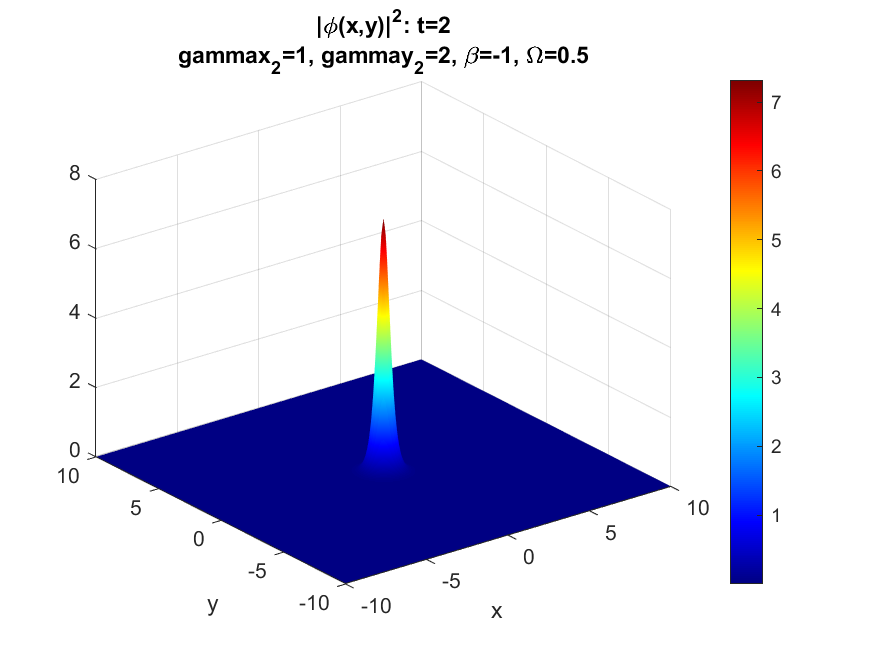} 
\includegraphics[width=0.48\textwidth]{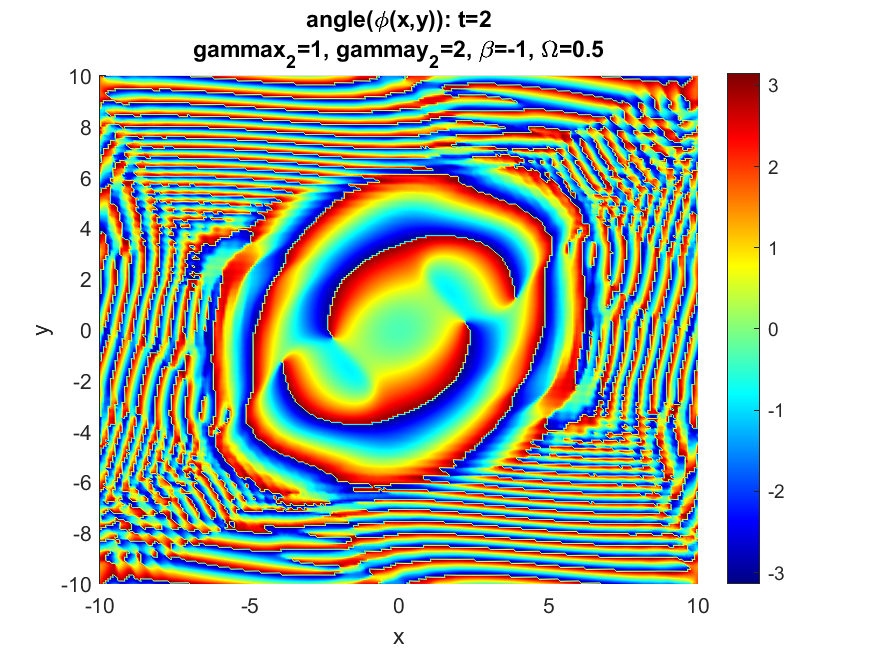}
\caption{\label{psi(t):belowQ0} Simulation of (\ref{psi:OmGa})  for $\psi_0=0.98Q_0$  at $t= 2$ with 
 $\Om=0.5$, $\ga=(1,\sqrt{2})$.  
  Left: plot of the square modular  $|\psi(t)|^2$.\    
Right: plot of the phase for $\psi(t)$}  
\label{f:belowQ:Om0.5-ga2(1,2)}
\end{figure}     

\begin{figure}[H]
\includegraphics[width=0.48\textwidth]{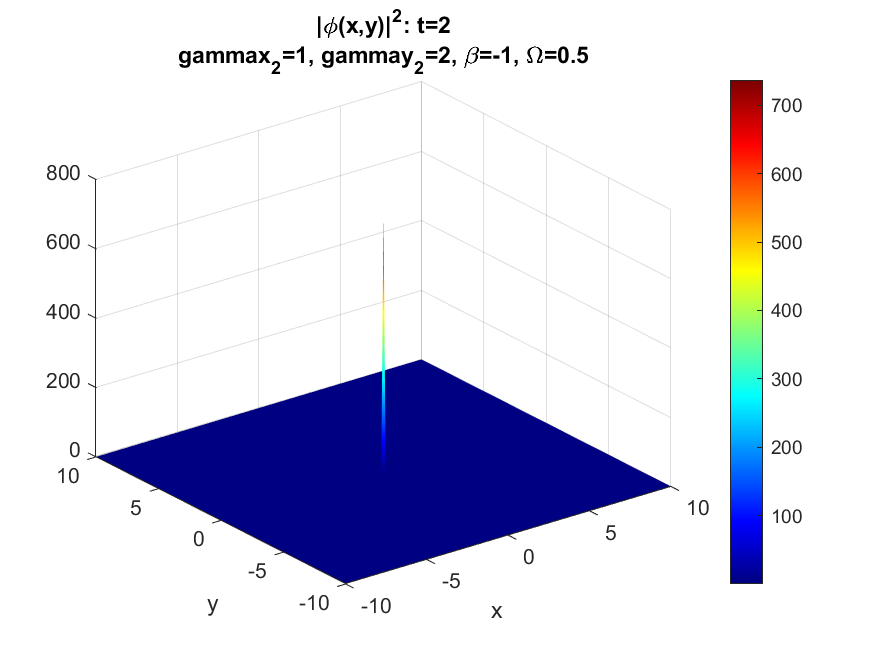}
\includegraphics[width=0.48\textwidth]{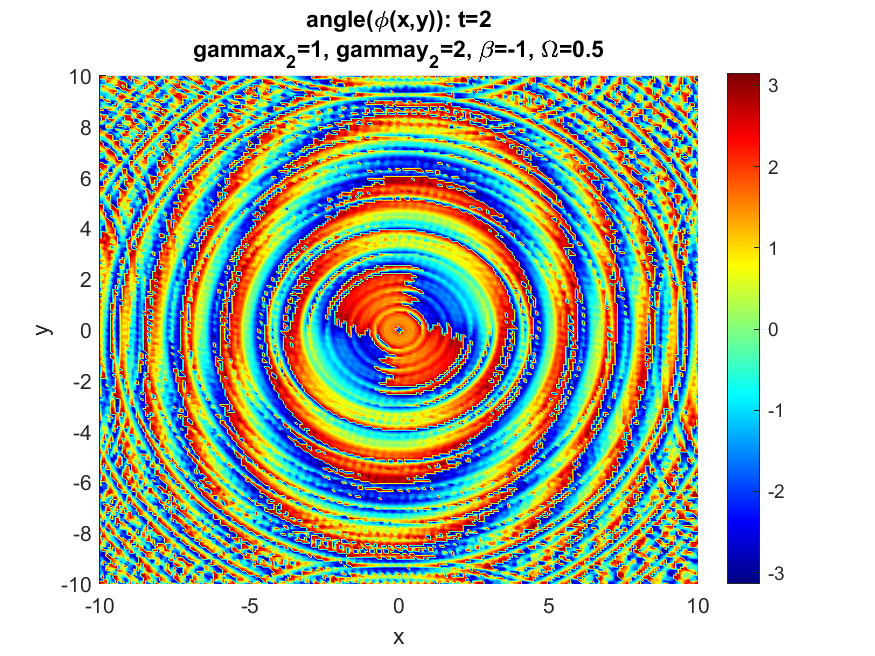}
\caption{\label{f:aboveQ0:Om0.5-ga2(1,2)}  
Blowup  for $\psi_0=1.02Q_0$  at $t= 2$ 
for  $\Om=0.5$, $\ga=(1,\sqrt{2})$.     
Left: plot of the square modular $|\psi(t)|^2$. 
Right: plot of the phase for $\psi(t)$}  
\end{figure}


 
\bigskip 
 
 \begin{figure}[H]
\includegraphics[width=0.48\textwidth]{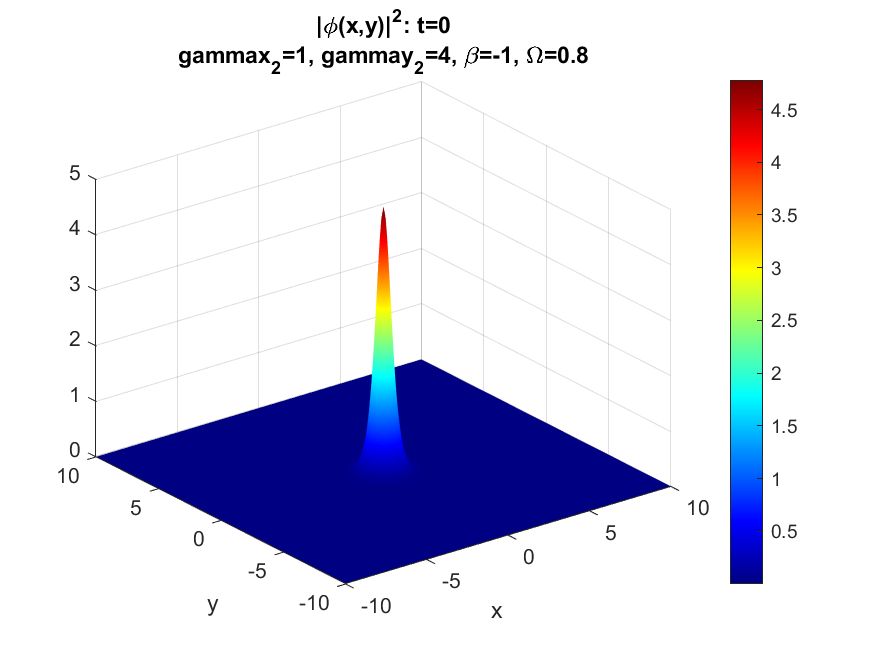}
\includegraphics[width=0.48\textwidth]{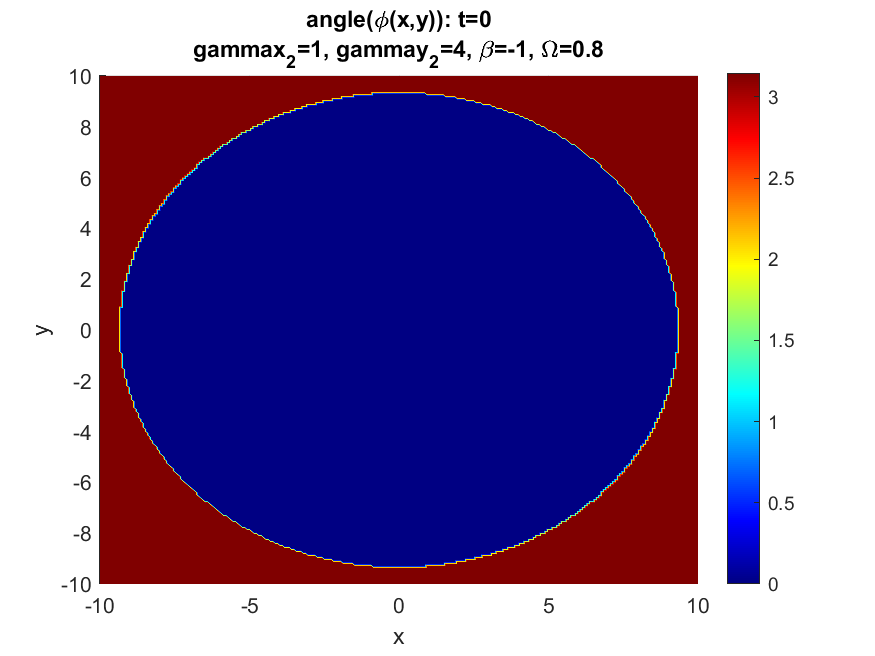}
\caption{Square modular and phase of 
scaled g.s.s. $\psi_0=0.99Q_0$} 
\label{fig:0.99Q}
\end{figure}

\bigskip

\begin{figure}[H] 
\includegraphics[width=0.48\textwidth]{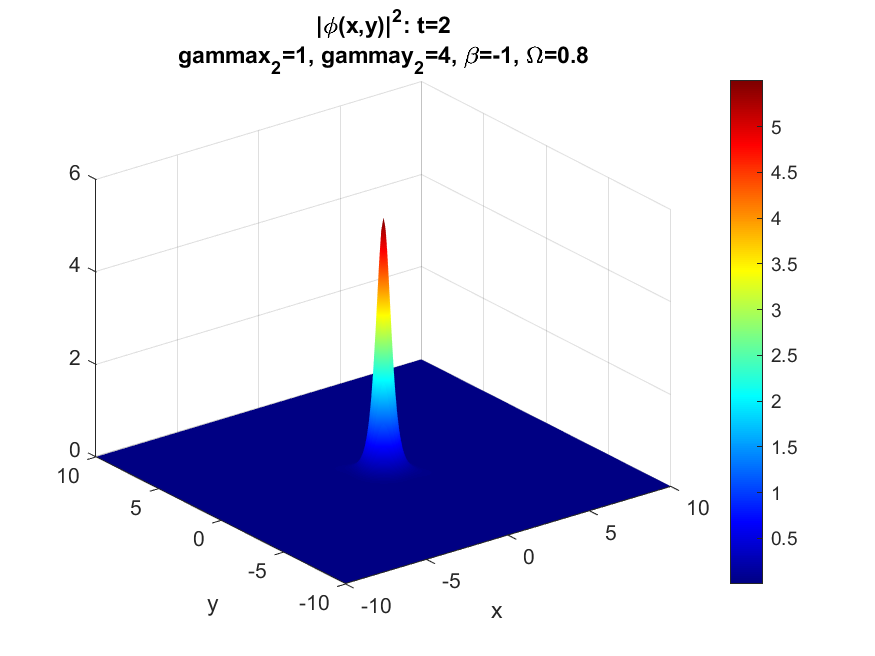}
\includegraphics[width=0.48\textwidth]{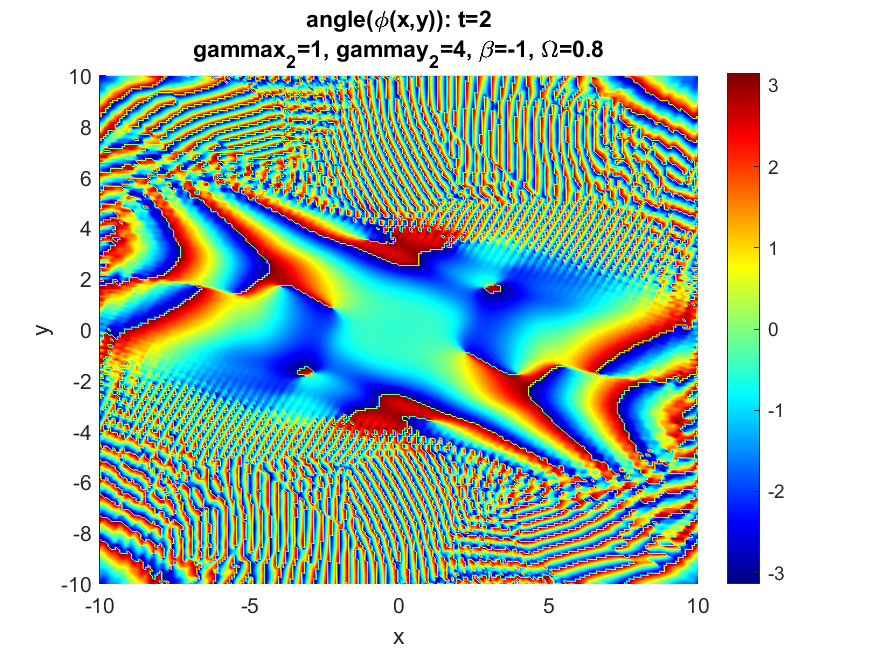} 
\caption{\label{f:0.99Q-global} 
Simulation of $\psi(t)$ at $t=2$ with   
   $\Om= 0.8$, $\ga=(1,2)$ and $\psi_0=0.99Q_0$} 
\end{figure}

\medskip
\begin{figure}[H]
\includegraphics[width=0.48\textwidth]{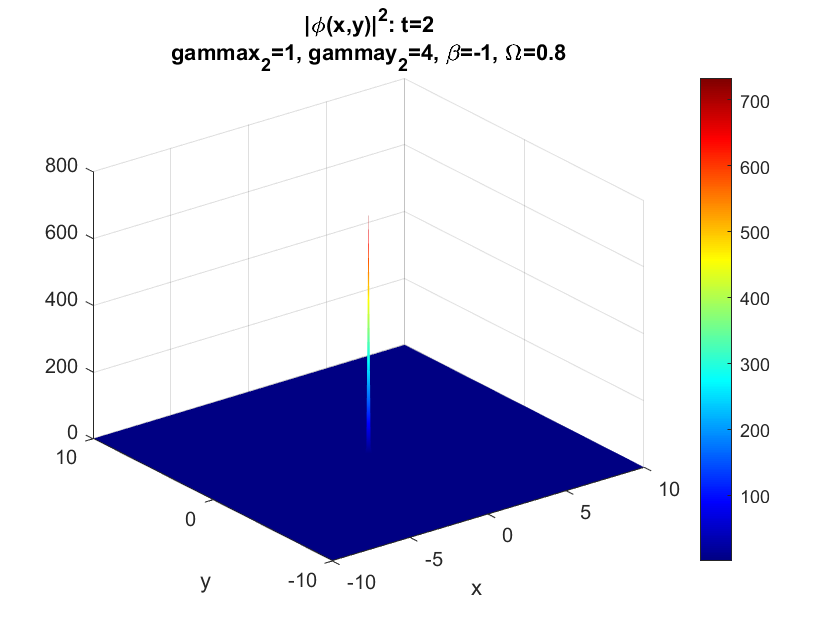}
\includegraphics[width=0.48\textwidth]{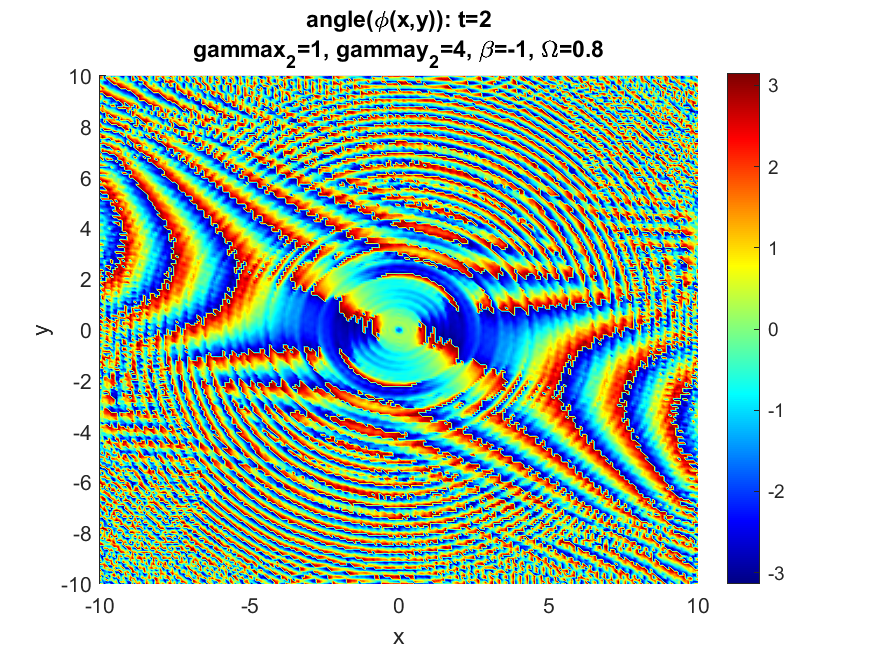}
\caption{\label{1.03Q:Om0.8-ga2(1,4)}  
Blowup  for $\psi_0=1.03Q_0$  at $t= 2$ 
for  $\Om=0.8$, $\ga=(1,{2})$.     
Left: plot of the square modular $|\psi(t)|^2$. 
Right: plot of the phase for $\psi(t)$} 
\end{figure} 


\bigskip 

\subsection{Scaled ground state $Q_{\Om,\ga}\, $} \label{ss:Q_OmV:simulat} 
 Repeat the routine above with $\psi_0= c Q_{\Omega,\gamma}$ in place of $cQ_0$
where $Q_{\Omega,\gamma}$   is the g.s.s. of RNLS (\ref{eQ:OmGa}). 
 In this case, since  $1=\int |Q_{\Omega,\gamma}|^2d x$ as prescribed in the coding, 
 the scaled mass $M(\psi_0)=c^2$.
Then, the minimal  mass blowup condition in Proposition \ref{t:gwp-blup-Om} requires 
 $ \int |\psi_0|^2\, dx >5.85043$. 
There is disparity in the use of scaled $Q=Q_0$ and $Q=Q_{\Om,\ga}$  to
 simulate the solutions for the RNLS in $\R^{1+2}$
\begin{equation}\label{psi:OmGa}  
 i\psi_t=-\frac{1}{2} \nabla^2 \psi+ V_\ga(x)\psi +\mu |\psi|^{p-1}\psi - i \Om x^\perp\cdot\nabla \psi \,,
	\end{equation} 
which is a two-dimensional  case of (\ref{RNLS_Omga}).  	

The computation and simulation for g.s.s. 
are run per the set of parameters $[\Om,\ga,c, Q]$ respectively given by: 
\begin{enumerate}
\item[(i)]  $ [0.5, (2,{8}), 1, Q_{\Om,\ga}]$, 
and  $[0.5, (2,{8}), 2.515, Q_{\Om,\ga}]$. 
\item[(ii)] $ [0.8, (1,\sqrt{2}), 2.415/2.43, Q_{\Om,\ga}]$ and $ [0.8, (1,{2}), 2.480/2.495, Q_{\Om,\ga}]$. 
\end{enumerate} 
The critical threshold value $c_0=\sqrt{ 5.85043}= 2.418766$. 
However, in the domain of $\psi_0=cQ_{\Om,\ga}$,
 the initial data may have to assume a mass value slightly larger than $c_0^2$
to give rise to a blowup solution of (\ref{psi:OmGa}). 
In what follows,  Figure \ref{fig:1QAV} is the simulation at $t=2$ of unscaled ground state where $\Om=0.5$, $\ga=(2,8)$.
It shows the standing wave solution for the RNLS (\ref{psi:OmGa}).  
\mbox{Figure \ref{f:2.515QAV}} to Figure \ref{t2:2.495QAV} are plots of the simulations of $\psi_0=cQ_{\Om,\ga}$ when $\Om=0.5$ or $0.8$.

\begin{figure}[H] 
\includegraphics[width=0.48\textwidth]{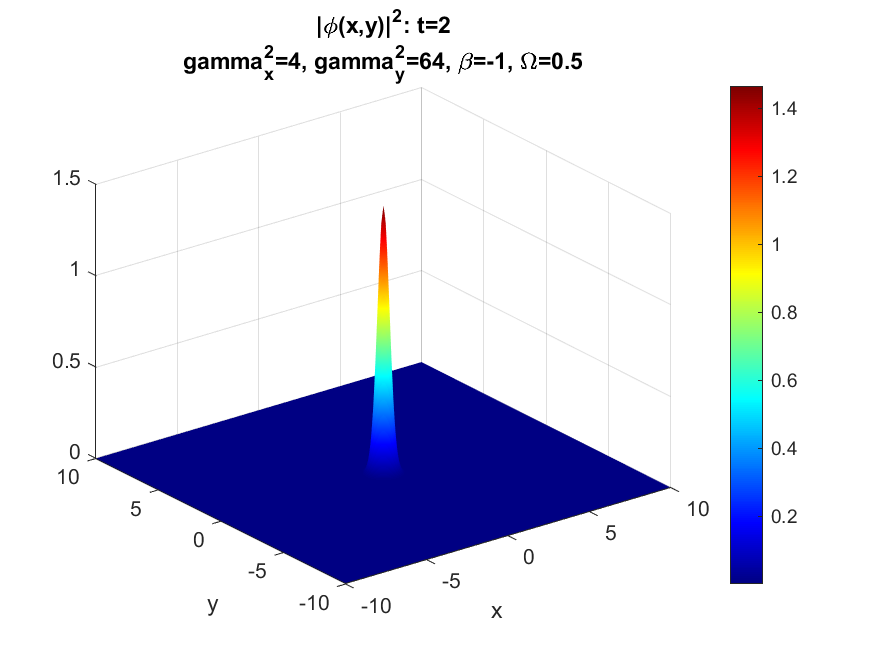}
\includegraphics[width=0.48\textwidth]{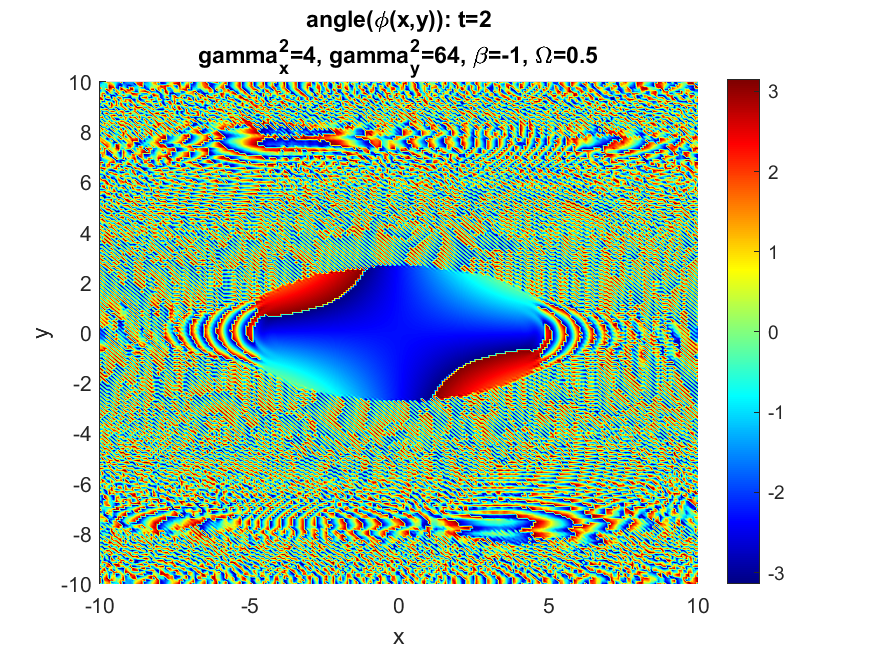}
\caption{Square modular and phase of 
unscaled g.s.s. $\psi_0=Q_{\Om,\ga}$  
with $\Omega=0.5$, $\ga_x=2, \ga_y=8$, $c=1$ at $t=2$
}\label{fig:1QAV}
\end{figure} 

\begin{figure}[H] 
\includegraphics[width=0.48\textwidth]{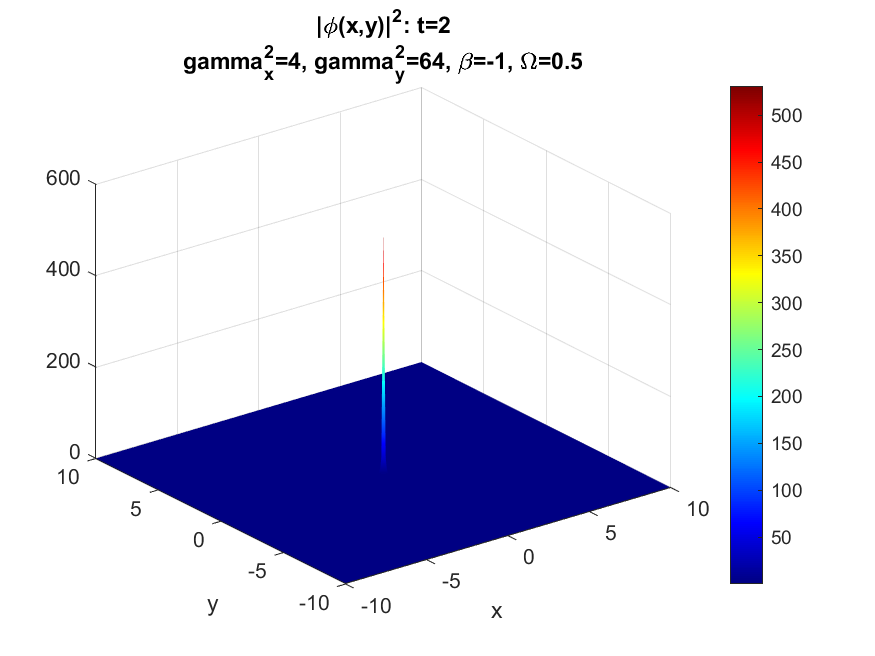}
\includegraphics[width=0.48\textwidth]{
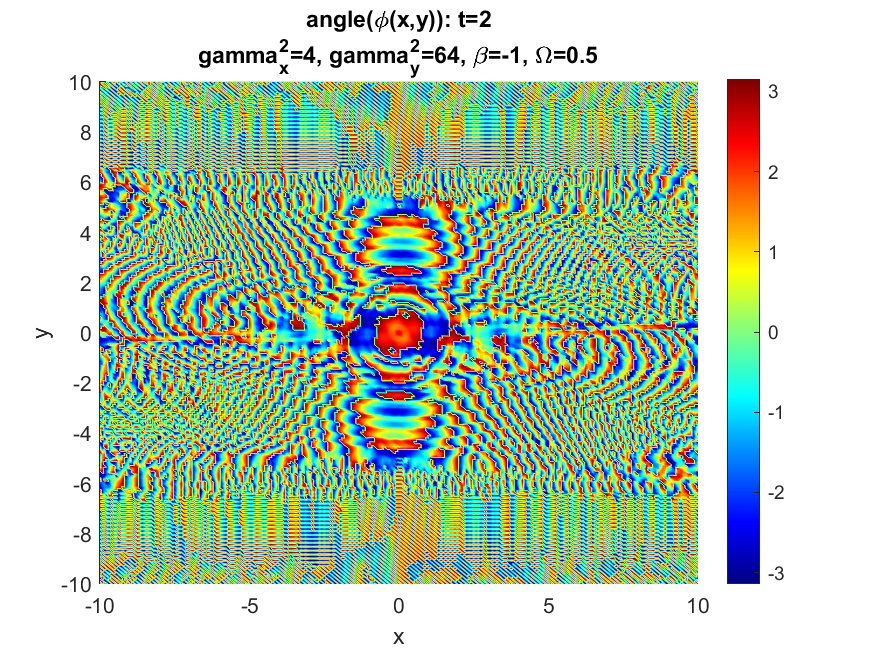}
\caption{Simulation of scaled ground state $\psi_0=2.515Q_{\Om,\ga}$  
with $\Omega=0.5$, $\ga_x=2, \ga_y=8$ 
}\label{f:2.515QAV}
\end{figure} 

\begin{figure}[H] 
\includegraphics[width=0.48\textwidth]{
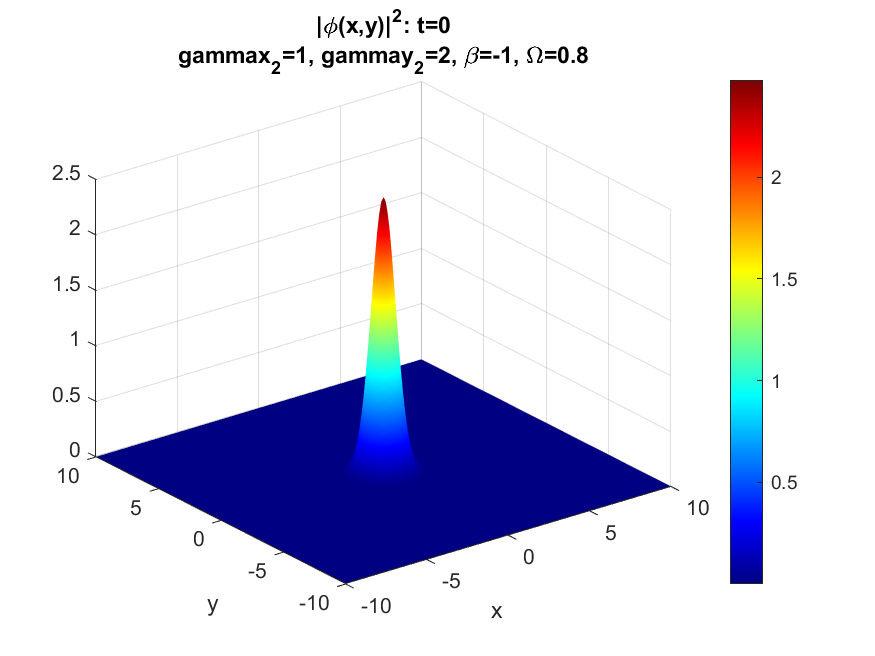}
\includegraphics[width=0.48\textwidth]{
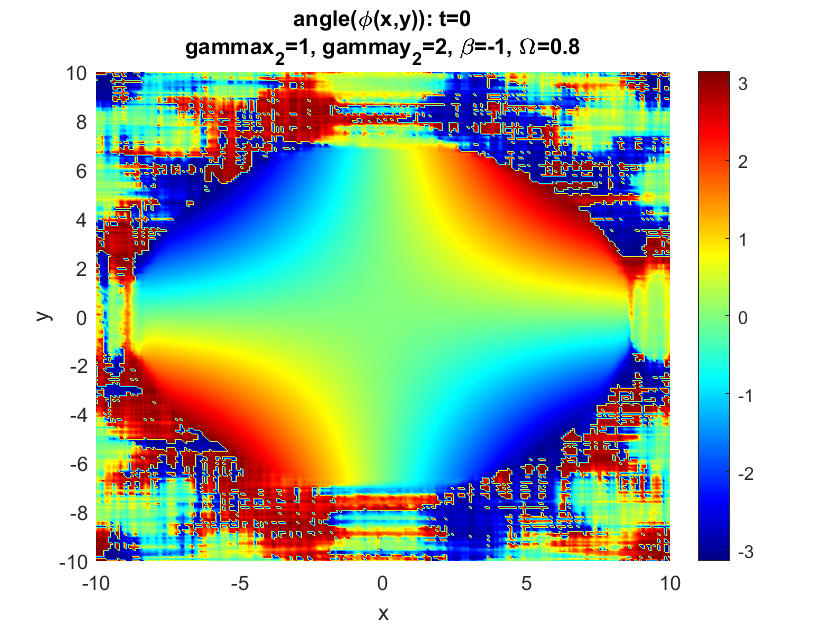}
\caption{Square modular and phase of 
scaled g.s.s. $\psi_0=2.415Q_{\Om,\ga}$  
with $\Omega=0.8$, $\ga_x^2=1, \ga_y^2=2$ 
}\label{fig:2.415Q_AV}
\end{figure}

\begin{figure}[H] 
\includegraphics[width=0.48\textwidth]{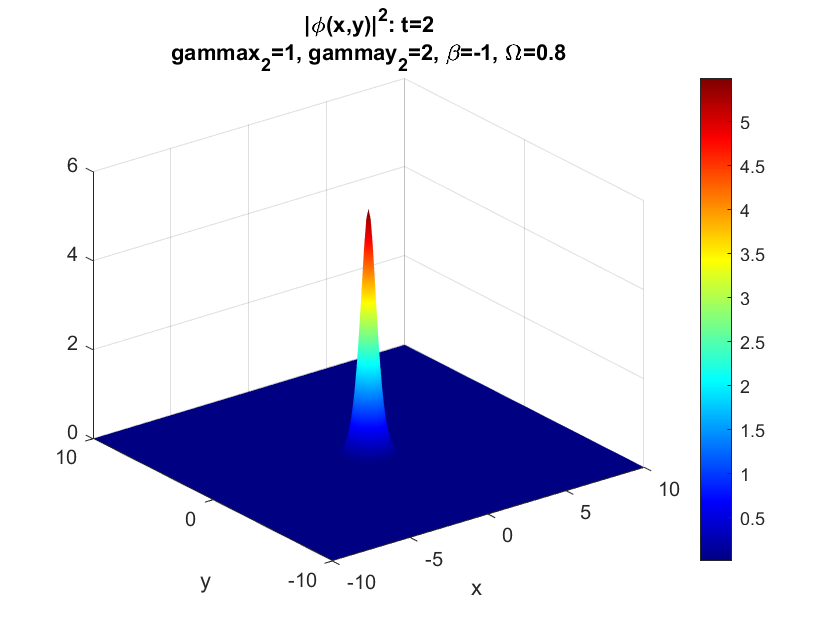}
\includegraphics[width=0.48\textwidth]{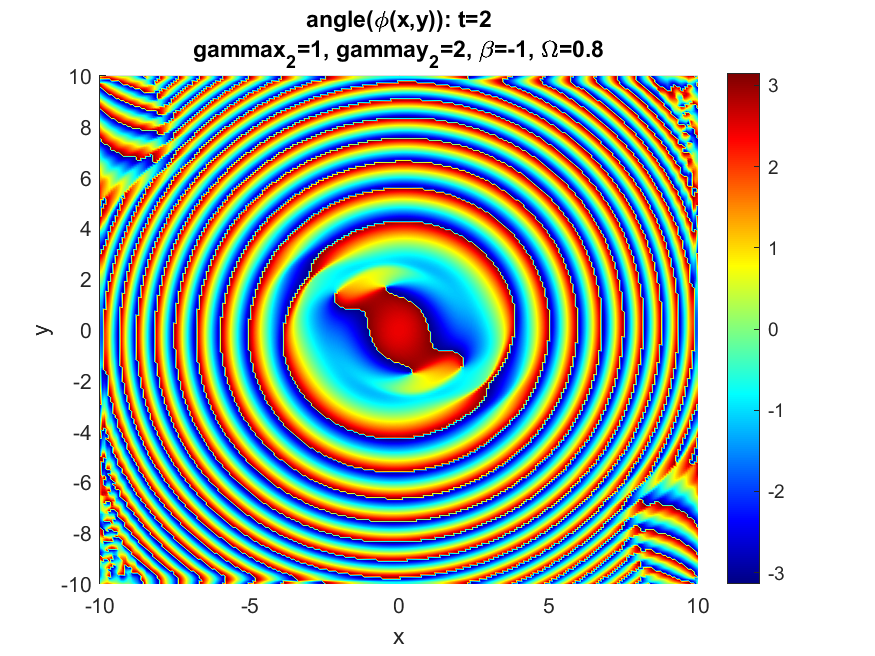}
\caption{Simulation for scaled g.s.s. $\psi_0=2.415Q_{\Om,\ga}$  at $t= 2$  
with $\Omega=0.8$, $\ga_x^2=1, \ga_y^2=2$ 
} 
\label{f:sim:2.415Q_AV}
\end{figure} 

\begin{figure}[H] 
\includegraphics[width=0.48\textwidth]{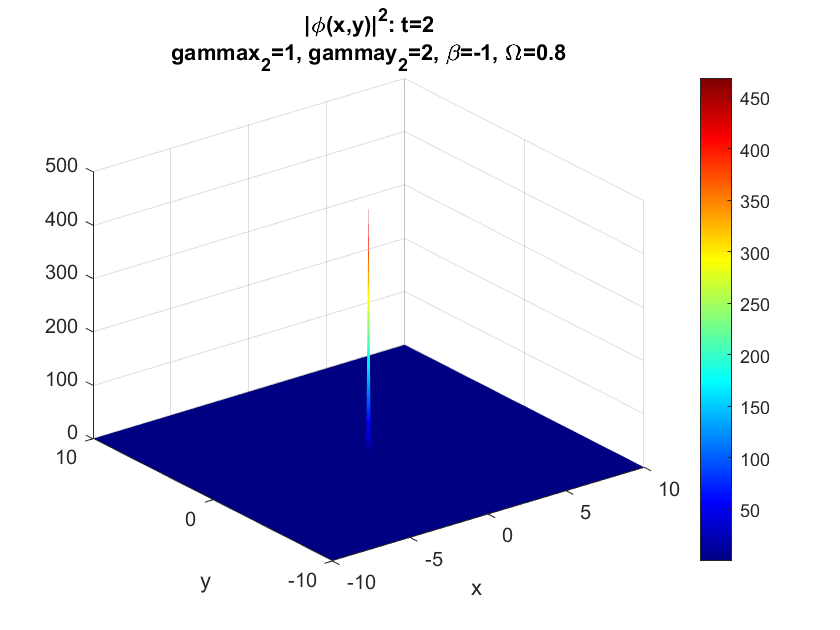}
\includegraphics[width=0.48\textwidth]{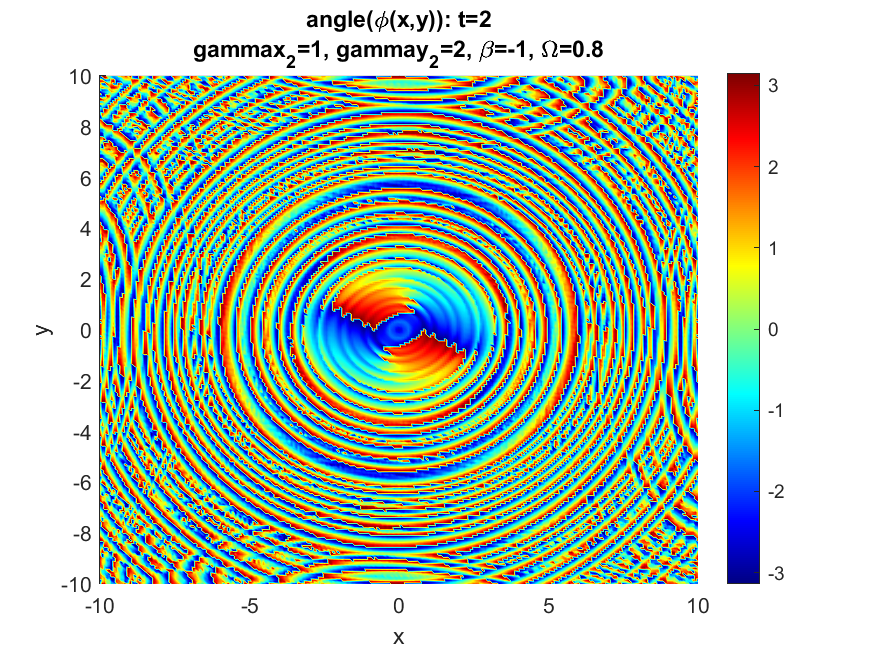}
\caption{Blowup for scaled g.s.s. $\psi_0=2.43 Q_{\Om,\ga}$  at $t= 2$  
with $\Omega=0.8$, $\ga_x^2=1, \ga_y^2=2$ 
} 
\label{sim:2.43QAV}
\end{figure} 


From Figures \ref{f:sim:2.415Q_AV}-\ref{sim:2.43QAV}, we see that when the anisotropy $(\ga_1,\ga_2)$ is small, then the threshold mass is  close to $5.83043$ in view of the result on 
$(2.415)^2=5.832225$ versus $(2.430)^2=5.9049$.  
However, when the anisotropy $(\ga_1,\ga_2)$ is larger, then the threshold mass seems to turn somewhat different!  
This can be observed by the following numerical result illustrated in Figures \ref{f:2.48QAV}-\ref{t2:2.495QAV}  
when the neighboring constant  
$(2.480)^2=6.1504$ and $(2.495)^2=6.225025$.  

\begin{figure}[H] 
\includegraphics[width=0.48\textwidth]{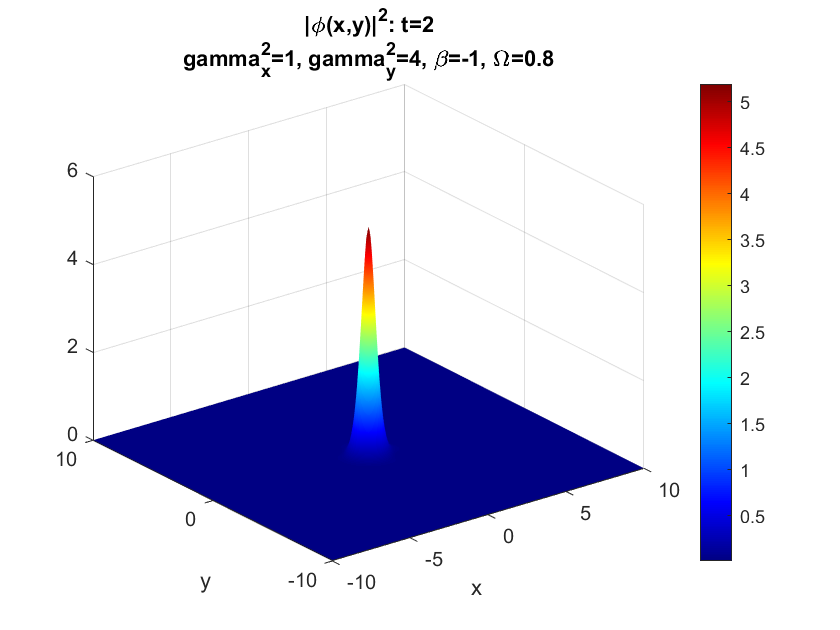}
\includegraphics[width=0.48\textwidth]{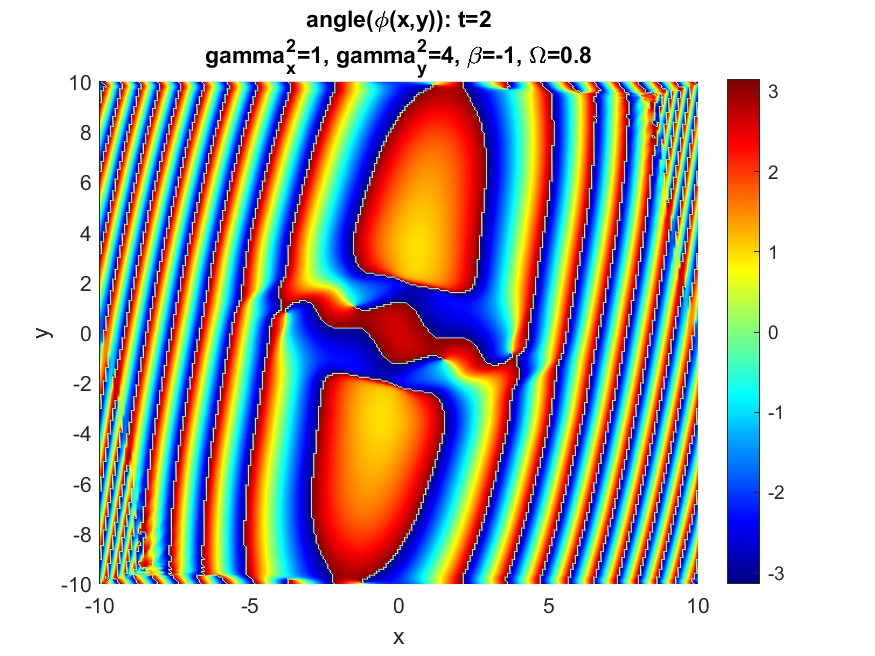}
\caption{Global existence for $\psi_0=c Q_{\Om,\ga}$  at $t= 2$  
with $\Omega=0.8$, $\ga_x^2=1, \ga_y^2=4$, $c=2.48$ 
}
\label{f:2.48QAV}
\end{figure}

\begin{figure}[H] 
\includegraphics[width=0.48\textwidth]{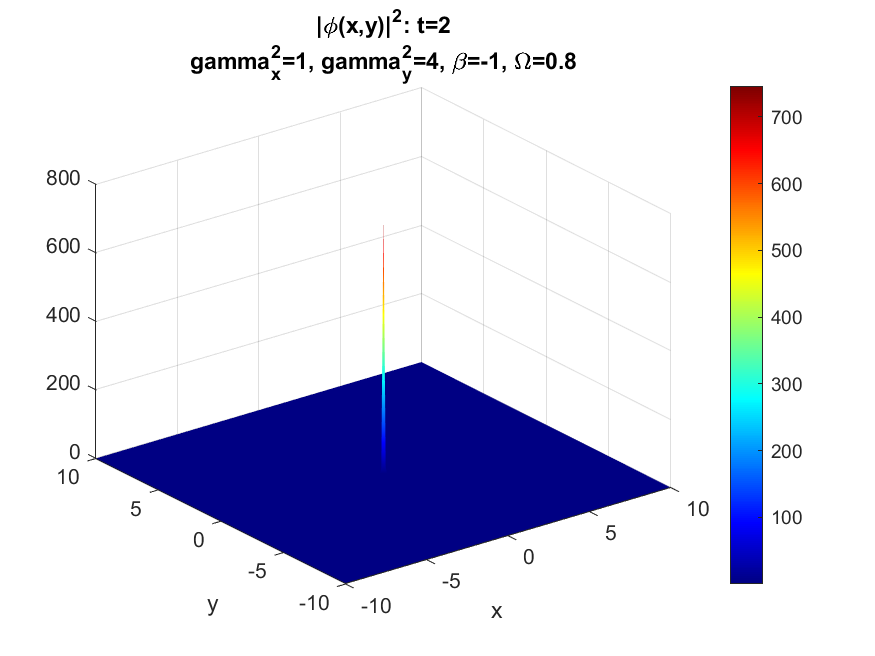}
\includegraphics[width=0.48\textwidth]{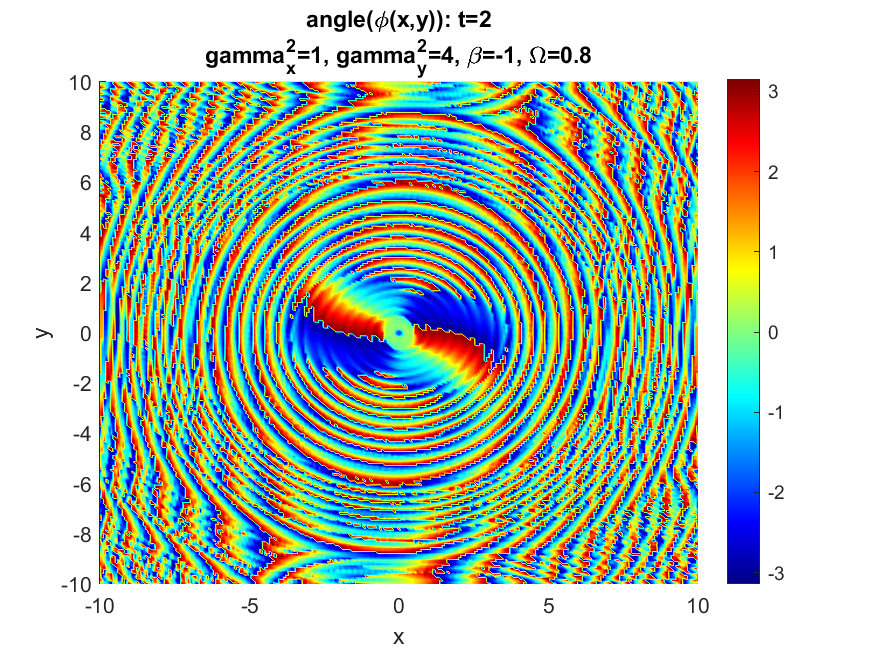}
\caption{Blowup for $\psi_0=c Q_{\Om,\ga}$  at $t= 2$  
with $\Omega=0.8$, $\ga_x^2=1, \ga_y^2=4$, $c=2.495$ 
} 
\label{t2:2.495QAV}
\end{figure} 


\subsection{Conclusion on the numerics}\label{ss:conclusion}  



The results in Subsection \ref{ss:Q_OmV:simulat} 
on the sharp value $c=c_{\Om,\ga}$ 
 provide a lower bound estimation about the mass of the g.s.s. $Q_{\Omega,\gamma}$.  
The simulations for the existence problem of ground states of (\ref{eQ:OmGa}) show that, in consistence with 
Theorem \ref{T-Existence} and Theorem \ref{t3:non-E},   
\begin{enumerate} 
\item[(i)] $0<\gamma_1<\gamma_2$ and $0<\Omega < \gamma_1\To$ existence of g.s.s.  
\item[(ii)]  $\gamma_1<\Omega<\gamma_2\To$ non-existence of g.s.s.  
\end{enumerate} 

The simulations for the threshold dynamics   problem for  (\ref{psi:OmGa}) support the following statement that extends 
 \cite[Theorem 1.1]{BHHZ23t} to the case of anisotropic potential under the condition $|\Omega| < \uga$: 
\begin{enumerate} 
\item[(i)] If $\norm{\psi_0}_2<\norm{Q_0}_2$,  then the solution flow $\psi_0\mapsto \psi(t)$ exists globally in time.   
\item[(ii)] Given any $c_0>\norm{Q_0}_2$,  there exists $\psi_0\in \Sigma$ with $\norm{\psi_0}_2=c_0$
such that the corresponding solution flow $\psi(t)$ of (\ref{psi:OmGa}) 
blows up in finite time.  
\end{enumerate} 


Note that Figures \ref{fig:2.415Q_AV} to \ref{t2:2.495QAV} 
indicate that when $\Om=0.8$, the threshold value $c_{thresh}=c_{\Om,\ga}$ may vary 
with the anisotropy parameter $\ga=(1,\sqrt{2})$ and $\ga=(1,2)$, namely, 
$c_{thresh}=2.43$ and $2.495$.  

If examing the case $\Om=0.5$,  then $c_{thresh}=2.45$ if $\ga=(1,2)$
while $c_{thresh}=2.515$ if $\ga=(2,8)$ 
in view of Figures \ref{Q:t2:pt5Ga12}-\ref{QAV:t2:2.45blup}
and Figures \ref{fig:1QAV}-\ref{f:2.515QAV}, respectively.  

 
 \begin{figure}[H] %
\includegraphics[width=0.48\textwidth]{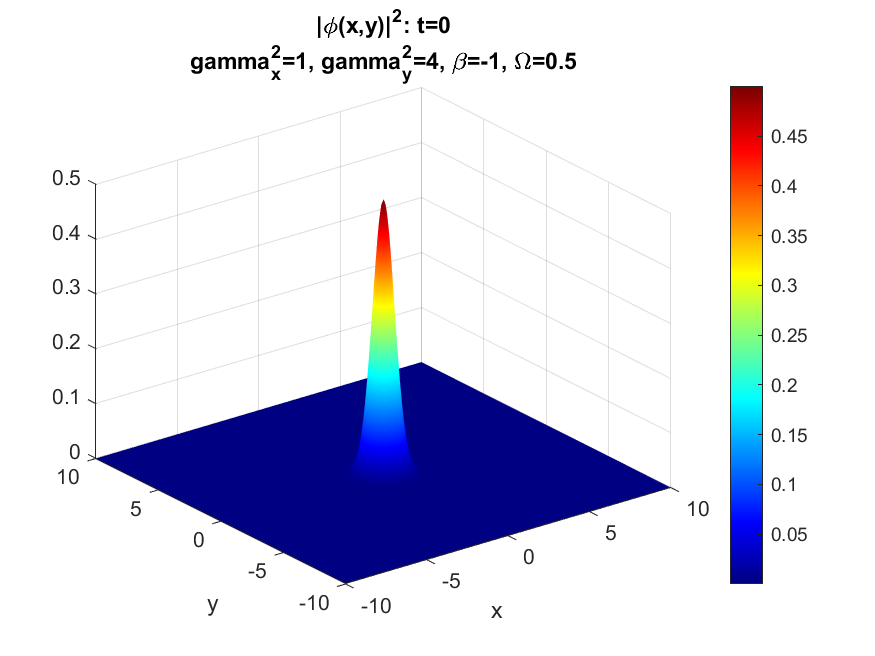}
\includegraphics[width=0.48\textwidth]{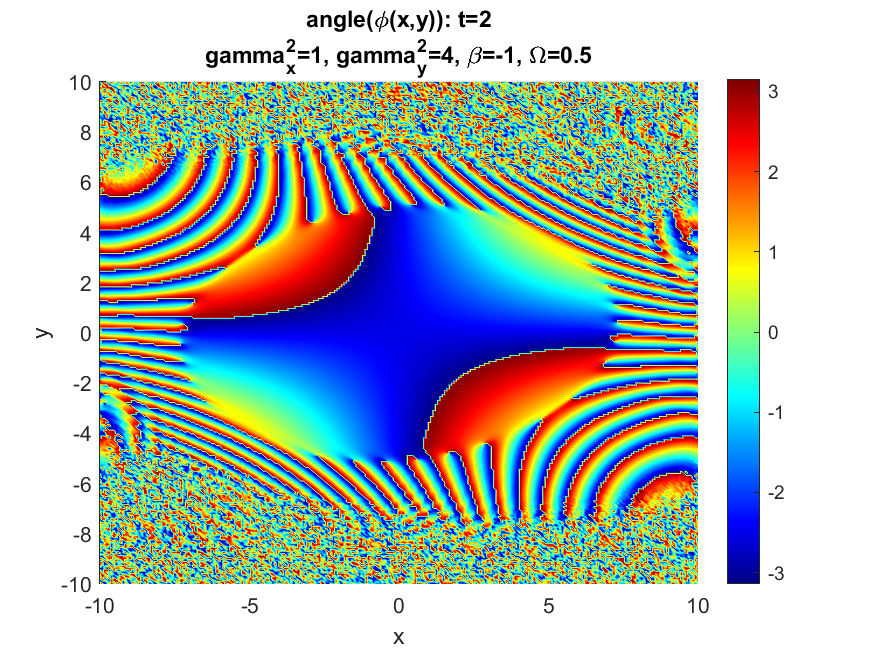}
\caption{square modular and phase  for unscaled g.s.s. $\psi_0= Q_{\Om,\ga}$  at $t= 2$  
with $\Omega=0.5$, $\ga_x=1, \ga_y=2$ 
} 
\label{Q:t2:pt5Ga12}
\end{figure} 
 
 \begin{figure}[H] %
\includegraphics[width=0.48\textwidth]{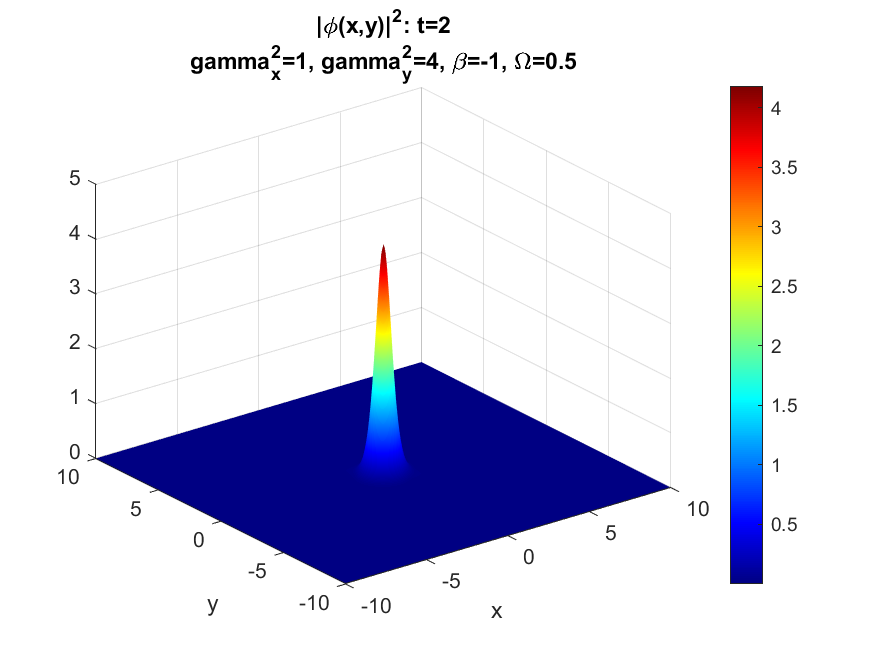}
\includegraphics[width=0.48\textwidth]{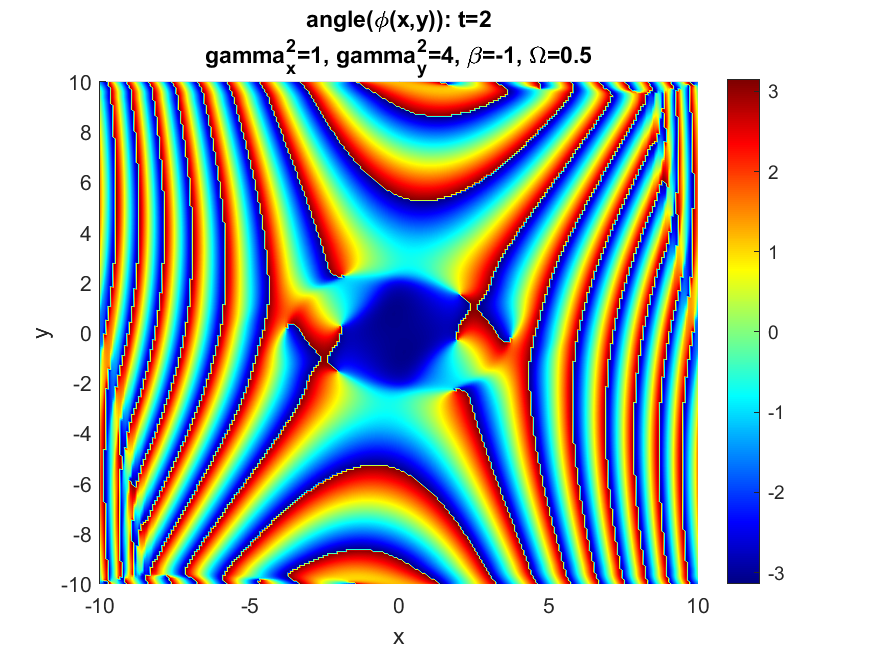}
\caption{Global existence for $\psi_0=c Q_{\Om,\ga}$  at $t= 2$  
with $\Omega=0.5$, $\ga_x=1, \ga_y=2$, $c=2.44$ 
} 
\label{QAV:t2:2pt44}
\end{figure}

\begin{figure}[H] %
\includegraphics[width=0.48\textwidth]{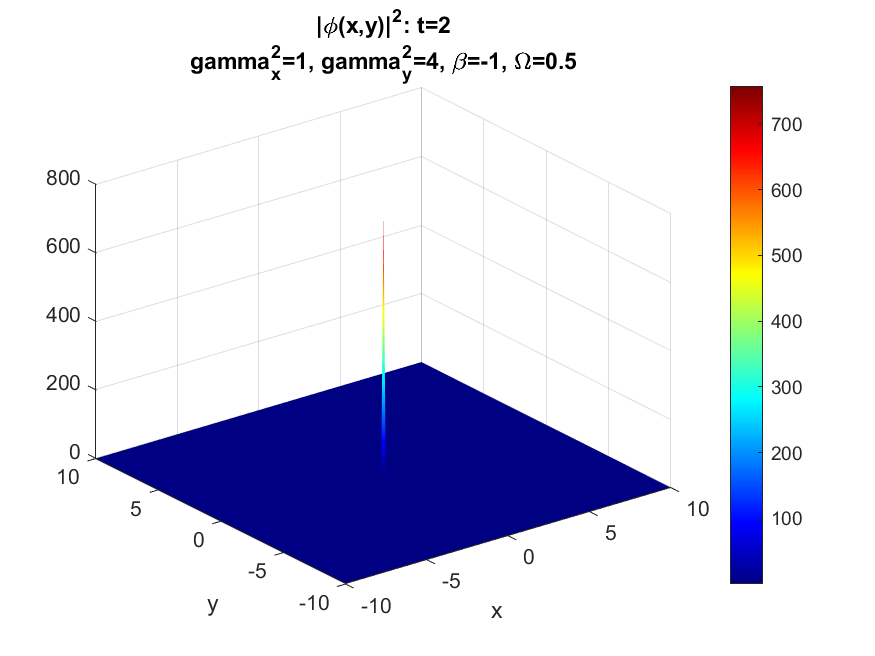}
\includegraphics[width=0.48\textwidth]{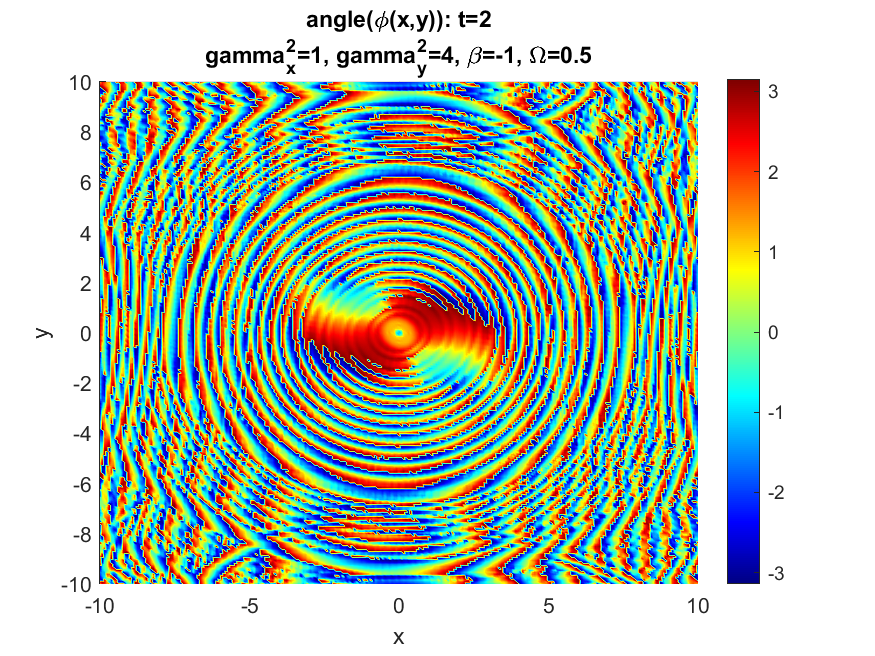}
\caption{Blowup for $\psi_0=c Q_{\Om,\ga}$  at $t= 2$  
with $\Omega=0.5$, $\ga_x=1, \ga_y=2$, $c=2.45$ 
} 
\label{QAV:t2:2.45blup}
\end{figure} 

 
 Concerning the construction and simulation for the ground states solution of (\ref{psi:OmGa}),  
 it would be of interest to further investigate the endpoint case $|\Om|=\uga$ which 
seems related to the partial confinement condition on $(A,V)$ for the RNLS. 
Such a problem might be a delicate and subtle issue in both theoretic and computational regimes.

\bigskip  


\mbox{} 
\bigskip 


\section{Appendix. Magnetic Sobolev space $\Sigma_{A,V}$} \label{magV:sobolev} 
Let $(A,V)$ be an electromagnetic potential such that 
  $ |A(x)|\le \al_0 (|V(x)|+1)^{1/2}$ for some $\al_0>0$.  
An alternative characterization of the energy space $\Sigma=\sH^1$ for  $\cL=H_{A,V}$ is given by 
\begin{align*}
\Sigma_{A,V}\coloneqq\{u\in H^1(\R^d):  \nabla_A u\in L^2,\  |V|^{1/2} u\in L^2 \}
\end{align*}
equipped with the norm 
\begin{equation*}
        \norm{u}_{\Sigma_{A,V}} \coloneqq \int \left( |\nabla_A u|^2 +  |V|\, |u|^2+|u|^2  \right)^{1/2} \,,
\end{equation*}
which is exactly (\ref{eAV:norm}), consult \cite[Proposition 4.4]{Z12a} for the case where $V(x)\approx \la x\ra^2$.  
The norm equivalence $\norm{u}_{\Sigma_{A,V}}\approx \norm{u}_{\sH^1}$ follows easily from the relation
\begin{align*}
& \int |(\nabla -iA)u |^2=\int |\nabla u |^2+\int |A|^2 |u|^2
- i \int Au\nabla \bar{u}+ i \int A\bar{u} \nabla u\\
=& \int |\nabla u |^2+\int |A|^2 |u|^2-2\Im \int A\bar{u} \nabla u\, ,
\end{align*} 
from which we obtain $ \norm{u}_{\Sigma_{A,V}}\ge \eps_0  \norm{u}_{\sH^1}$ for some $\eps_0\in (0,1)$. 
In passing we note that the divergence theorem yields 
\begin{align*}
2\Re (A\nabla u, u)=&-\int (\dive A)|u|^2 \\
2\Re (iA\nabla u, u)=&-2\Im (A\nabla u, u ). 
\end{align*}
Hence the operator $iA\cdot\nabla $ is selfadjoint if and only if $\dive A=0$. 



Some intrinsic norm characterization is given in \cite{NguyenSqua18magSob} for $V=0$ and 
$A$: $\R^d\to \R^d$ lipschitz.  
The existence of ground states for fractional magnetic laplacian $(-\De_A)^s$ is considered in \cite{dASqua18gss-mag}
using a similar norm.  
The second part of the appendix discusses the relations between spaces $(\dot{H}^1, H^1)$ and $(H^1_A,\Sigma_A)$. 
\begin{proposition}\label{p:H_A-H1}  Assume $A=A(x)=(A_1,\dots,A_d)$ satisfies the condition 
\begin{equation}\label{eA:dA}
2 A\left(\frac{{\bf x}}{2}\right) +(DA){\bf x}={\bf 0}. 
\end{equation}
where ${\bf  x}=\la x_1,\dots,x_d\ra^T\in \R^d$.
E.g. $A=Mx$ where $M$ is skew-symmetric in $\R^{d^2}$.  
Then $e^{ix\cdot A\left(\frac{{\bf x}}{2}\right)} u\in H^1\iff  u\in \dot{H}^1$.
 Or, equivalently, 
 $ u\in H^1\iff e^{-ix\cdot A\left(\frac{x}{2}\right)} u\in \dot{H}^1$. 
Moreover, the transform $u\mapsto e^{ix\cdot A(\frac{{\bf x}}{2})} u$
leaves invariant the weighted space $\Sigma_A=H_A^1$ under condition (\ref{eA:dA}).
\end{proposition}

\begin{example}
    Let $A= Mx$ with $M^T=-M$. 
Let $\om=\om(x)= x\cdot A\left(\frac{x}{2}\right)$. 
\end{example} 
\begin{proof}
Let $u\in \dot{H}^1$.  We compute with $\om=\om(x)={x\cdot A\left(\frac{x}{2}\right)} $
\begin{align*}
 \nabla  (e^{ix\cdot A(\frac{x}{2})} u )
=& i e^{i \om(x)  } A\left(\frac{x}{2}\right) u +  (e^{ix\cdot A\left(\frac{x}{2}\right)} u) \left(\frac{i}{2} \right) (DA)x\ 
+ e^{i\om} \nabla u \\ 
=&  e^{i\om} (\nabla u ),
\end{align*}
where $DA=\begin{pmatrix} & (\frac{\pa A}{\pa x_1})^T \\ 
& \cdots\\
&  (\frac{\pa A}{\pa x_d})^T
\end{pmatrix}$,
the  vector $(\frac{\pa A}{\pa x_j})^T=\la \frac{\pa A_1}{\pa x_j},\dots,\frac{\pa A_d}{\pa x_j}\ra$ is the $j$th
 row vector. 
 The second statement in the proposition follows from $\Sigma_A=H_A^1$ (see \cite{AHS78mag,Dinh22mag3Drev})  and
\begin{align*}
 (\nabla -iA) (e^{ix\cdot A(\frac{x}{2})} u )
=&  e^{i\om} (\nabla -iA )u\,,
\end{align*}
where $\nabla_A$ is the co-variant derivative. 
Noted that the space $\Sigma_A=\{u\in L^2\cap S^\prime(\R^d):  \nabla u\in L^2, |A| u\in L^2(\R^d) \}$ is contained in $H_A^1=\Sigma_A\subset H^1$ under the condition (\ref{eA:dA}).
\end{proof}
 
\begin{remark}
   If $A$ is linear, i.e., $A=Cx$ for any matrix $C$,
 then it is easy to observe by a straightforward calculation that \eqref{eA:dA} is verified if and only if $C$ is skew-symmetric, $C^T=-C$. 
\end{remark}
 
However if \eqref{eA:dA} is not verified, then 
$H_A^1\ne H^1$ as the following example shows.

\begin{example}
    Let  $A=(x_2,x_1)$.  Then the assertion $\Sigma_A= H_A^1$ in Proposition \ref{p:H_A-H1} may fail.  
\end{example}
  In fact,  take any $u\in H^1\setminus \Sigma_A$, then $v=e^{i\om(x) } u $ is in  $H_A^1$ but not in  $\Sigma_A$.    
In this case we compute 
\begin{align*}
 (\nabla -iA) (e^{i x_1x_2} u )
=&  e^{i\om} (\nabla u). 
\end{align*}
In general, let $A=Cx$, $C$ non anti-symmetric such that 
\begin{equation}
2 A\left(\frac{{\bf x}}{2}\right)+ (DA){\bf x} - 2A({\bf x})=  {\bf 0} . \label{eC-A}
\end{equation}
Then in $\R^d$ with $\om(x)=x\cdot A(x/2)$
\begin{align*}
 (\nabla -iA) (e^{i \om(x)} u )
=&  i e^{i\om} u \left(  A\left(\frac{{\bf x}}{2}\right)+  \frac12(DA){\bf x} -A\right)     +e^{i\om(x)} (\nabla u)\\
=&e^{i\om(x)} (\nabla u)\in L^2\,,
\end{align*}
which suggests that $w=e^{i \om(x)} u $ is in $H_A^1\setminus H^1$.
The inclusion picture here is that $\Sigma_A\subset H_A^1\cap H^1$ but $H_A^1$ and $H^1$ do not contain one other.

\begin{remark}
    If $A=Cx$ is linear, then it is elementary to check that   
(\ref{eC-A}) holds true if and only if  $C=(c_{ij})_{d\times d}$ is a real symmetric matrix.  
  Let $A=\la y,x\ra$, then $f(x,y)=e^{ixy}\phi(x,y)$ with $\phi\in H^1\setminus \Sigma$
verifies:  $\nabla_A f\in L^2(\R^2)$.  
However for $A=\la -y,x\ra$ then
$f= e^{i\theta(x,y)} \phi$ won't belong to $H_A^1$ for any $\theta$.
\end{remark}

\bigskip 

\nd {\bf Acknowledgments} The authors thank Dr. Kai Yang for  valuable discussions on the numerical computation for the 
ground state solutions as well as relevant references. 
C.L. is partially supported by  NSF Grant DMS-1818684. 

\bibliographystyle{plain} 
\bibliography{mNLS0331_ALZ}

\end{document}